\documentclass[reqno]{amsart}

\usepackage{amsmath,amssymb,amsthm,amsfonts,
	enumerate,hyperref,cleveref}

\hypersetup{
	colorlinks=true,
	linkcolor=blue,
	filecolor=magenta,      
	urlcolor=cyan,
	pdftitle={Overleaf Example},
	pdfpagemode=FullScreen,
}
\usepackage{dsfont}
\usepackage[all]{xy}
\usepackage[T1]{fontenc}
\usepackage{tikz}
\usepackage{pgfplots}
\pgfplotsset{compat=1.15}
\usepackage{mathrsfs}
\usetikzlibrary{arrows}

\theoremstyle{plain}

\newtheorem{thrm}{Theorem}[section]
\newtheorem{cor}[thrm]{Corollary}
\newtheorem{prop}[thrm]{Proposition}
\newtheorem{lem}[thrm]{Lemma}

\newtheorem{probl}[thrm]{Problem}

\theoremstyle{definition}
\newtheorem{defn}[thrm]{Definition}
\newtheorem{rem}[thrm]{Remark}

\crefrangeformat{equation}{#3(#1)#4--#5(#2)#6}

\crefname{thrm}{Theorem}{Theorems}
\crefname{theorem}{Theorem}{Theorems}
\crefname{lem}{Lemma}{Lemmas}
\crefname{cor}{Corollary}{Corollaries}
\crefname{prop}{Proposition}{Propositions}
\crefname{defn}{Definition}{Definitions}
\crefname{exm}{Example}{Examples}
\crefname{rem}{Remark}{Remarks}
\crefname{conj}{Conjecture}{Conjectures}
\crefname{quest}{Question}{Questions}
\crefname{section}{Section}{Sections}
\crefname{equation}{\unskip}{\unskip}
\crefname{enumi}{\unskip}{\unskip}
\crefname{subsection}{Subsection}{Subsections}

\newcommand{\CC}{\mathbb{C}}

\begin{document}
	\title[Preservers of truncations of triple products]{Maps preserving the truncation of triple products on Cartan factors}	
	
   \author[J. J. Garc\'{e}s]{Jorge J. Garc\'{e}s}
\address[J. J. Garc\'{e}s]{Departamento de Matem{\'a}tica Aplicada a la Ingenie{\'r}a Industrial, ETSIDI, Universidad Polit{\'e}cnica de Madrid, Madrid, Spain.}
\email{j.garces@upm.es}

\author[L. Li]{Lei Li}
\address[L. Li]{School of Mathematical Sciences and LPMC, Nankai University, 300071 Tianjin, China.}
\email{leilee@nankai.edu.cn}

\author[A.M. Peralta]{Antonio M. Peralta}
\address[A.M. Peralta]{Instituto de Matem{\'a}ticas de la Universidad de Granada (IMAG), Departamento de An{\'a}lisis Matem{\'a}tico, Facultad de
	Ciencias, Universidad de Granada, 18071 Granada, Spain.}
\email{aperalta@ugr.es}

\author[S. Su]{Shanshan Su}
\address[S. Su]{School of Mathematics, East China University of Science and Technology, 200237 Shanghai, China. (Current address) Departamento de An{\'a}lisis Matem{\'a}tico, Facultad de
Ciencias, Universidad de Granada, 18071 Granada, Spain.}
\email{lat875rina@gmail.com}

	
\begin{abstract} Let $\{C_i\}_{i\in \Gamma_1},$ and $\{D_j\}_{j\in \Gamma_2},$ be two families of Cartan factors such that all of them have dimension at least $2$, and consider the atomic JBW$^*$-triples $A=\bigoplus\limits_{i\in \Gamma_1}^{\ell_{\infty}} C_i$ and $B=\bigoplus\limits_{j\in \Gamma_2}^{\ell_{\infty}} D_j$. Let $\Delta :A \to B$ be a {\rm(}non-necessarily linear nor continuous{\rm)} bijection preserving the truncation of triple products in both directions, that is, $$\begin{aligned}
	\boxed{a \mbox{ is a truncation of } \{b,c,b\}} \Leftrightarrow \boxed{\Delta(a)  \mbox{ is a truncation of } \{\Delta(b),\Delta(c),\Delta(b)\}}
 \end{aligned}$$ Assume additionally that the restriction of $\Delta$ to each rank-one Cartan factor in $A$, if any, is a continuous mapping. Then we show that $\Delta$ is an isometric real linear triple isomorphism. We also study some general properties of bijections preserving  the truncation of triple products in both directions between general JB$^*$-triples. 
	\end{abstract}
	
	\maketitle
	
	\tableofcontents
	
\section{Introduction}

Problems on preservers constitute a prolific and stimulating domain, where different disciplines, like algebra, geometry, and functional analysis interplay. Surprisingly, there is a unifying approach which allows to consider several scattered problems under a common optics. Let us fix a general background with two mathematical structures $A$ and $B$ over linear spaces admitting a certain binary  (algebraic) operation $p_{_A} (a,b)$ and $p_{_B} (c,d)$, respectively, and certain partial relations $\sim_{_A}$ and $\sim_{_B}$, respectively. Suppose $\Delta : A\to B$ is a (non-necessarily linear nor continuous) \emph{bijection preserving the relations $\sim_{_A}$ and $\sim_{_B}$ on products}, that is, $$\boxed{a\sim_{_A} p_{_A} (b,c) \hbox{ in } A } \Rightarrow \boxed{\Delta(a)\sim_{_B} p_{_B} (\Delta(b),\Delta(c)) \hbox{ in } B }.$$ When the implication ``$\Rightarrow$'' is replaced by the equivalence ``$\Leftrightarrow$'' we say that $\Delta$ \emph{preserves the relations} $\sim_{_A}$ and $\sim_{_B}$ \emph{on products in both directions}. The problem is to determine appropriate relations $\sim_{_A}$ and $\sim_{_B}$ and products $p_{_A} (\cdot,\cdot)$ and $p_{_B} (\cdot,\cdot)$ to conclude that $\Delta $ is a linear mapping and understand its precise form. \smallskip

For example, by taking $A = B(H)$ and $B= B(K)$, the von Neumann algebras of all bounded linear operators on complex Hilbert spaces $H$ and $K$ with their natural associative products (i.e. $p_{_A} (a,b) = p_{_B} (a,b) = a b$) and the relations $a\sim_{\lambda} b$ defined by $a$ being a $\lambda$-Aluthge transform of $b$ for a fixed $\lambda \in ]0,1[$, we find the problem studied by F. Chabbabi in \cite{Chabb2017}, where he concluded that if $\Phi : B(H) \to B(K)$ is a bijection and dim$(H)\geq 2$, $\Phi$ preserves the relation ``being a $\lambda$-Aluthge transform'' for a fixed $\lambda \in ]0,1[$ on products of operators if, and only if, there exists a unitary operator $u : H\to K$ such that $\Phi (a) = u a u^*$ for all $a\in B(H)$.  When the associative products of $B(H)$ and $B(K)$ are replaced with the Jordan product $p_{_A} (a,b) = p_{_B} (a,b)=a\circ b = \frac12 (a  b +ba )$, F. Chabbabi and M. Mbekhta proved that a bijection  $\Phi : B(H) \to B(K)$ with  dim$(H)\geq 2$, preserves the relation being a $\lambda$-Aluthge transform for a fixed $\lambda \in ]0,1[$ on Jordan products of operators if, and only if, there exists a unitary or an anti-unitary operator $u : H\to K$ such that $\Phi (a) = u a u^*$ for all $a\in B(H)$ (see \cite{ChabbMbekhta2017, ChabbMbekhta2022corrigendum}). The case in which $A$ and $B$ are von Neumann algebras, $A$ admits no abelian direct summand, $\sim_{_A}=\sim_{_B}=\sim_{\lambda}$ is the relation ``being a $\lambda$-Aluthge transform of'' (for a fixed $\lambda\in [0,1]$) and $p_{_A} (a,b) = p_{_B} (a,b)=a\circ b^*$ (respectively $p_{_A} (a,b) = p_{_B} (a,b)=a b^*$) the problem is considered in \cite{EssPe2018}, where it is established that every bijective mapping $\Phi:A\to B,$ preserving $\sim_{\lambda}$ for products of the form $a\circ b^*,$ maps the hermitian part of $A$ onto the hermitian part of $B$ and its restriction $\Phi|_{M_{sa}} : M_{sa}\to N_{sa}$ is a Jordan isomorphism.\smallskip

There is an earlier precedent in the literature which admits an even simpler statement. Suppose $M$ is a von Neumann algebra {\rm(}or more generally an AW$^{*}$-algebra{\rm)} which has no abelian direct summand, and $N$ is a C$^{*}$-algebra. A result by J. Hakeda shows that taking $p(a,b) = a b$ and $a \sim b \Leftrightarrow a = b$  on $M$ and $N,$ then a bijection $\Phi : M\to N$ satisfying $\Phi (x^*) = \Phi (x)^*$ ($\forall x\in M$), and preserving equality for products must be a real linear $^*$-isomorphism (see  \cite{Hakeda1986}). Furthermore J. Hakeda and K. Sait{\^o} proved that under the same hypotheses but replacing the associative product by $p(a,b) = a\circ b$, the bijection $\Phi$ must be a real linear Jordan $^*$-isomorphism \cite{HakedaSaito1986,Hakeda1986Jordan}. \smallskip

Related results do not only restrict to C$^*$-algebras. Y. Friedman and J. Hakeda proved the next outstanding result for JB$^*$-triples (see subsection~\ref{subsec: definitions} for definitions). Let $M$ be a JBW$^*$-triple with no abelian direct summand, and let $N$ be another JBW$^*$-triple. Take $\sim$ as equality on $M$ and $N$ and set $p(a,b) = \{ a, b, a\}$. Then each bijection $\Delta: M\to N$ preserving the relation equality for the product $p(a,b) = \{a,b,a\}$ is additive (see \cite{FriedmanHakeda1988}).\smallskip

In a more recent contribution (see \cite{Jia_Shi_Ji_AnnFunctAnn_2022}), X. Jia, W. Shi, and G. Ji consider a variant of our problem in the case in which $A=B=B(H)$, the relations $\sim_{_A}$ and $\sim_{_B}$ are given by $a\sim_{t} b$ if $a$ is a truncation of $b$ (i.e. $aa^* a = ab^* a$) and the products $p(a,b) = a b $ and $p(a,b) = a b a$. In the first case they are called preservers of truncation of products, while in the second one they are known as preservers of truncation of Jordan products of operators. Assuming that dim$(H)\geq 2,$ they prove that a bijection $\Delta$ on $B(H)$ preserves the truncation of products of operators in both directions if, and only if, there exists a unitary or an anti-unitary operator $u$ on $H$ such that $\Delta(a) = u a u^*$ for any $a \in B(H)$ \cite[Theorem 2.1]{Jia_Shi_Ji_AnnFunctAnn_2022}. Furthermore, by \cite[Theorem 3.1]{Jia_Shi_Ji_AnnFunctAnn_2022}, a bijection $\Delta: B(H)\to B(H)$ preserves the truncation of triple Jordan products of operators in both directions if, and only if, there exist a constant $\lambda\in  \{1,-1\}$ and a unitary or an anti-unitary operator $u$ on $H$ such that one of the following statements holds:\begin{enumerate}[$(a)$]
\item $\Delta(a) = \lambda u a u^*$ for all  $a\in B(H)$;
\item $\Delta(a) = \lambda u a^* u^*$ for all  $a\in B(H)$,
\end{enumerate} consequently, $\Delta$ is a real linear triple automorphism on $B(H)$ for the triple product given by $\{a,b,c\}:= 2^{-1} (a b^* c + c b^*a),$ which is the natural triple product of $B(H)$ when the latter is regarded as a JB$^*$-triple or Cartan factor of type 1. There are other 5 types of Cartan factors distinguishable from the point of view of holomorphic theory and functional analysis, and of course we can consider those JB$^*$-triples obtained as $\ell_{\infty}$-sums of families of Cartan factors (see \cref{subsec: definitions}). \smallskip

In this paper we begin by observing that the relation ``being a truncation of'' makes perfect sense in the general (and widder) setting of JB$^*$-triples. We say that an element $a$ in a JB$^*$-triple $E$ is a truncation of another element $b$ in $E$ ($a\sim_{t} b$ in short) if $\{a,a,a\} = \{a,b,a\}$. This definition agrees with the usual one on $B(H)$. We devote \cref{sec: truncations} to present the general properties of the relation ``being a truncation of'' in general JB$^*$-triples. It should be remarked that when restricted to the lattice of tripotents this relation coincides with the usual partial ordering (see \cref{rem trunc for tripotents}). \smallskip

In this study, we are mainly interested in the study of those (non-necessarily linear nor continuous) bijections $\Delta$ between JB$^*$-triples $E$ and $F$ preserving the relation $\sim_{t}$ on products of the form $p(b,c) = Q(b) (c)=\{b,c,b\}$ in both directions, which are called \emph{preservers of the truncation of triple products} (see \cref{sec: preservers of the trunc of triple products}). Observe that on $B(H)$ our product has the form $p(b,c) = bc^* b$, which differs from the one employed in \cite{Jia_Shi_Ji_AnnFunctAnn_2022}. There is a natural reason for this, the space $B(H,K)$, the Banach space of all bounded linear operators between two complex Hilbert spaces $H$ and $K$, is a JB$^*$-triple for the triple product $\{a,b,c\}:= 2^{-1} (a b^* c + c b^*a),$ however expresions of the form $bcb$ are not computable on $B(H,K)$. We devote \cref{sec: preservers of the trunc of triple products} to study the general properties of bijections preserving the truncation of triple products in both directions between general JB$^*$-triples, JB$^*$-algebras and spin factors.\smallskip

Our main conclusion is presented in \cref{sec: main conclusions}, where we establish that if $A$ and $B$ are atomic JBW$^*$-triples non-containing $1$-dimensional Cartan factors, every {\rm(}non-necessarily linear nor continuous{\rm)} bijection $\Delta :A \to B$ preserving the truncation of triple products in both directions, that is, $$\begin{aligned}
	\boxed{a \mbox{ is a truncation of } \{b,c,b\}} \Leftrightarrow \boxed{\Delta(a)  \mbox{ is a truncation of } \{\Delta(b),\Delta(c),\Delta(b)\}}
 \end{aligned}$$ and satisfying that its restriction to each rank-one Cartan factor in $A$ is continuous must be an isometric real linear triple isomorphism (see \cref{t main thrm preservers of truncations of triple products}). For the sake of brevity, we shall simply comments that the technical arguments rely on proving that every bijection between JB$^*$-triples preserving the truncation of triple products in both directions defines a bijection preserving order and orthogonality in both directions between the lattices of tripotents of both JB$^*$-triples (see \cref{lem pres tripotents and leq}). Thanks to this property we can apply a recent result by F. Friedman and the third author of this note showing that each bijection preserving the partial ordering in both directions and orthogonality in one direction between the lattices of tripotents of two atomic JBW$^*$-triples not containing rank-one Cartan factors, which is additionally continuous at a tripotent in the domain whose projection onto every Cartan factor is non-zero, extends to a real linear triple automorphism between the JBW$^*$-triples \cite[Theorem 6.1]{Fried_Peralta_Ann_Math_Phys_2022}. Another tool developed in this note is an identity principle for mappings between atomic JBW$^*$-triples established in \cref{thm of extension T and Delta}, which seems interesting by itself.\smallskip

We cannot conclude this introduction without noting that the study of maps preserving certain properties related to truncations is currently being intensively explored.  J. Yao and G. Ji proved in \cite{Yao_Ji_JMathResAppl_2022} that an additive and surjective mapping $\Delta: B(H)\to B(H)$ (where dim$(H)\geq2$) preserves the truncation of operators in both directions if, and only if, there exist a nonzero scalar $\alpha\in \mathbb{C}$ and operators $u,v$ on $H$ which are both unitary or anti-unitary such that $\Delta (a)=\alpha u a v,$ for all $a\in B(H)$, or $\Delta(a)= \alpha u a^* v,$ for all $a\in B(H)$. Another interesting contribution by Y. Mao and G. Ji (see \cite{MaoJi2024}) establishes that a bijection $\Delta: B(H)\to B(H)$ (with dim$(H)\geq2$) preserves truncations of operators in both directions if, and only if, there exist a nonzero scalar $\alpha\in \mathbb{C}$ and operators $u,v$ on $H$ which are both unitary or anti-unitary such that $\Delta (a)=\alpha u a v,$ for all $a\in B(H)$, or $\Delta(a)= \alpha u a^* v,$ for all $a\in B(H)$. Our contribution in this note sets the links to study a great variety of problems in the widder setting of Cartan factors and JB$^*$-triples.

\subsection{Definitions and terminology}\label{subsec: definitions} \ \smallskip

As we have commented at the introduction, JB$^*$-triples provide a formal frame to study many Jordan models including C$^*$-algebras, JB$^*$-algebras, Cartan factors and spin factors among others. A \emph{JB$^*$-triple} (see \cite{Kaup_RiemanMap}) is a Banach space whose norm adapts perfectly with an additional algebraic structure given by a continuous triple product $$ \{\cdot, \cdot, \cdot  \}: E \times E \times E \to E : (x, y,z ) \mapsto \{ x,y,z \}$$
for all $x,y,z \in E$, which is linear in the first and third variables, conjugate-linear in the middle one, and, additionally, satisfies the following axioms: 
\begin{enumerate}[$(i)$]
\item\label{Jordan identity} For each $a,b$ in $E$, the operators $L(a,b): E\to E$, $L(a,b)(c) := \{a,b,c\}$ satisfy the identity
$$L(w,v)\{x,y,z \} = \{L(w,v)x,y,z \} - \{x,L(v,w)y,z \}+\{x,y, L(w,v)z\},$$ for all $x,y,z,w,v \in E$; \hfill ($\emph{Jordan identity}$)  
\item For each $x\in E$, the operator $L(x,x)$ is hermitian with non-negative spectrum;
\item $\| \{x,x,x\}  \| = \| x\|^3$ for all $x \in E$.\hfill ($\emph{Gelfand-Naimark axiom}$) 
\end{enumerate}

Some special operators given in terms of the triple product are defined as follows: for $a,b \in E$, $Q(a,b)$ is the conjugate-linear operator on $E$ defined by $Q(a,b)(z) := \{ a,z,b \} (\forall\, z \in E)$. Note that both $L(a,b)$ and $Q(a,b)$ are bounded. For the sake of simplicity, we simply write $Q(a)$ for $Q(a,a)$, and $a^{[3]}$ for $\{a,a,a\}$.\smallskip

A triple homomorphism (respectively, triple isomorphism) between JB$^*$-triples is a linear (respectively, linear and bijective) mapping preserving triple products.\smallskip 

Well known classes of algebras are encompassed in the category of JB$^*$-triples. The best two known examples are C$^*$-algebras and JB$^*$-algebras, both of which are JB$^*$-triples with respect to triple products defined by 
\begin{equation}\label{eq C*-triple product}
\{ a,b,c \} = \frac{1}{2} (ab^*c+ cb^*a)
\end{equation}
and  
\begin{equation}\label{eq JB$^*$-triple product}
\{ a,b,c \} = (a \circ b^*) \circ c -(a \circ c) \circ b^* + (c \circ b^* ) \circ a,
\end{equation} respectively (see \cite[pages 522, 523 and 525]{Kaup_RiemanMap} and \cite[Theorem 3.3]{BraunKaupUpmeier78}).\smallskip

In the setting of JB$^*$-triples, idempotents and projections do not make any sense, the role of these elements is played by tripotents. An element $e$ in a JB$^*$-triple $E$ is called a \emph{tripotent} if $L(e,e) (e) = \{e,e,e\}= e$. Observe that in a C$^*$-algebra $A,$ equipped with the triple product in \eqref{eq C*-triple product}, tripotents correspond to partial isometries.\smallskip

Each tripotent $e$ in $E$ induces a \emph{Peirce decomposition} of $E$ corresponding to the direct sum:
\begin{equation}
   E = E_{0}(e) \oplus E_{1}(e) \oplus E_{2}(e), 
\end{equation}
where each $E_{j}(e) := \{  a\in E: L(e,e)(a) = \frac{j}{2} a \}$ is a subtriple of $E$ called the \emph{Peirce $j$-subspace} ($j = 0,1,2$). Triple products among elements in Peirce subspaces follow certain patterns known as \emph{Peirce rules} or  \emph{Peirce arithmetics}, namely, $\{ E_{i}(e),E_{j}(e),E_{k}(e) \} \subseteq E_{i-j+k}(e)$ for all $i,j,k = \{0,1,2\}$ and 
$$ \{E_{0}(e), E_{2}(e), E \} = \{ E_{2}(e), E_{0}(e),E \}  = 0,$$
where $E_{i-j+k}(e) = \{0\} $ if $i-j+k \neq  \{0,1,2\}$. The Peirce 2-subspace $E_{2}(e)$ is actually a unital JB$^*$-algebra with identity $e$, where the Jordan product and involution operations are given by $a \circ_{e} b := \{a,e,b \}$ and $a ^{*_{e}} := \{e,a,e \}$, respectively (cf. \cite{Horn_MathScand_1987} or \cite[Fact 4.2.14, Proposition 4.2.22, and Corollary 4.2.30]{CabreraPalaciosBook}).\smallskip

Illustrative examples of JB$^*$- and JBW$^*$-triples are given by the so-called \textit{Cartan factors} defined as follows: 
Suppose that $H$ and $K$ are two complex Hilbert spaces. The space $B(H, K)$ of all bounded linear operators between $H$ and $K$ are called \textit{Cartan factors of type $1$}.
In order to describe the next two Cartan factors, let $j: H\to H$ be a conjugation (that is, a conjugate linear isometry of period $2$). The subspaces $\{a\in B(H): a=-ja^*j\}$ and $\{a\in B(H): a=ja^*j\}$ of $B(H)$ are called \textit{Cartan factors of type $2$} and $3$, respectively. 
A Banach space $X$ is called a \textit{Cartan factor of type $4$} or \textit{spin factor} if $X$ admits a complete inner product $(\cdot|\cdot)$ and a conjugation $x\mapsto \bar{x}$, for which the norm of $X$ is given by
\[\|x\|^2=(x|x)+\sqrt{(x|x)^2-|(x|\bar{x})|^2}\] and the triple product of $X$ is defined by $$\{x, y, z\} = \langle x|y\rangle z + \langle z|y\rangle  x -\langle x|\overline{z}\rangle \overline{y}, \ \ (x,y,z\in X).$$

The finite dimensional JB$^*$-triples $M_{1,2}(\mathbb{O})$ and $H_3(\mathbb{O}),$ of all $1\times 2$ matrices of (complex) octonions $\mathbb{O}$, and the $3\times 3$ hermitian matrices with entries in $\mathbb{O}$, respectively, are called \textit{Cartan factor of type $5$} and $6$, respectively, which are also called the \textit{exceptional Cartan factors} (see  \cite{HamKalPe2023} for a more detailed presentation).\smallskip

We say that a tripotent $e$ is \emph{(algebraically) minimal} (respectively, \emph{complete} or \emph{maximal}) if $E_{2}(e) = \CC e \neq \{ 0\}$ (respectively, $E_{0}(e) = \{ 0\}$). Let the symbols $\mathcal{U}(E)$,  $\mathcal{U}_{min}(E)$ and $\mathcal{U}_{max}(E)$ stand for the sets of all tripotents, all minimal tripotents, and maximal tripotents in $E$, respectively. There is an additional subclass of tripotents given by unitaries. A tripotent $u$ in $E$ is called unitary if $E_2(u) =E$, in which case $E$ is a unital JB$^*$-algebra with unit $u$. The set $\mathcal{U}_{max}(E)$ admits a remarkable geometric characterization since it coincides with the set of all extreme points of the closed unit ball of $E$ (see, for example, \cite[Corollary 4.8]{EDW_RUTT_JLMS_1988} or \cite[Theorem 4.2.34]{CabreraPalaciosBook}). Consequently, by the Krein-Milman theorem, every JB$^*$-triple enjoying the additional property of being a dual Banach space contains an abundant collection of tripotents. JB$^*$-triples which are dual Banach spaces are called \emph{JBW$^*$-triples}. JBW$^*$-triples enjoy additional properties, for example, each one of them, admits a unique isometric predual and its triple product is separately weak$^*$ continuous (see \cite{Bar_Tim_MathScand_1986}). It is known that the bidual of each JB$^*$-triple is JBW$^*$-triple (see \cite{Dineen_theseconddual_1986}). \smallskip

Let $x,y$ be elements in a JB$^*$-triple $E$. According to the standard notation,  $x$ and $y$ are called \emph{orthogonal} ($x \perp y$ in short) if $L(x,y) = 0$; the reader is referred to \cite[Lemma 1]{BFPGMP08} for additional properties, for example, that the relation ``being orthogonal'' is symmetric. It should be noted that two tripotents, $w,v\in \mathcal{U}(E)$, are orthogonal if and only if $w \in E_{0}(v)$. A natural partial ordering on $\mathcal{U}(E)$ can be established through orthogonality. Given  $e,u\in \mathcal{U}(E),$ we write $e \leq u$ if $u-e \in \mathcal{U}(E) $ and $u -e \perp e$, or equivalently, $e$ is a projection in the unital JB$^*$-algebra $E_{2}(u)$. We note that this relation of orthogonality among elements in $\mathcal{U}(E)$ agrees with the original notion of orthogonality of C$^*$-algebras in the case that $E$ is a C$^*$-algebra. We refer the reader to \cite{Battaglia_1991,FriedRusso_Predual} for a good account of this part. \smallskip

A particular subclass of JBW$^*$-triples is the one of all atomic JBW$^*$-triples. A JBW$^*$-triple $M$ is called \emph{atomic} if the linear span of all minimal tripotents in $M$ is w$^*$-dense in $M$. Every non-zero tripotent in an atomic JBW$^*$-triple can be written as the supremum of a family of mutually orthogonal minimal tripotents in it. The type I von Neumann factor $B(\mathcal{H}),$ of all bounded operators on (complex) Hilbert space $\mathcal{H}$, is an example of atomic JBW$^*$-triple. Minimal tripotents or partial isometries in $B(H)$ are of the form $\xi \otimes \eta$ with $\xi$ and $\eta$ in the unit sphere of $\mathcal{H}$. In fact, all Cartan factors are atomic JBW$^*$-triples. Furthermore, every atomic JBW$^*$-triple can be decomposed as the direct $\ell_{\infty}$-sum of a family of Cartan factors, and every JB$^*$-triple embeds isometrically as a JB$^*$-subtriple of an atomic one (see \cite[Proposition 2 and Theorem E]{FriedRusso_GelfNaim}).\smallskip

Suppose that $a$ is an element in a JBW$^*$-triple $E$. Then there exists a smallest tripotent $e$ in $E$ satisfying that $a$ is a positive element in the JBW$^*$-algebra $E_{2} (e)$. This tripotent is called the \emph{range tripotent} of $a$ (in $E$), and is denoted by $r(a)$.  It is known that $r(a)$ also coincides the range projection of $a$ in $E_2 (r(a))$. If $E$ is a mere JB$^*$-triple, it might contain no tripotents. However, its second dual, $E^{**}$, contains an abundant collection of tripotents. The range tripotent of each element $a$ in $E$ is computed in $E^{**}$. Furthermore, if $F$ is any JB$^*$-subtriple of $E$ containing $a$, the range tripotent of $a$ in $F^{**}$ is precisely the range tripotent of $a$ in $E^{**}$, that is, $r(a)$ does not change when computed with respect to any JB$^*$-subtriple containing $a$. The greatest tripotent  $u(a)$ in $E^{**}$ such that $Q(u(a))(a) = u(a)$ is known as the \emph{support tripotent} of $a$. These tripotents satisfy  $u(a) \leq a \leq r(a)$ in the natural partial ordering of the JBW$^*$-algebra $E^{**}_{2} (r(a))$ (see \cite[Section 3]{EDW_RUTT_JLMS_1988} for more details).\smallskip

Let us finally recall that the JB$^*$-subtriple $E_a$ generated by a single element $a$ in a JB$^*$-triple $E$ is JB$^*$-triple isometrically isomorphic to some commutative C$^*$-algebras admitting $a$ as a positive generator (cf. \cite[Corollary 1.15]{Kaup_RiemanMap}, \cite{Kaup_Sing_Val} and \cite[Theorem 4.2.9]{CabreraPalaciosBook}). It follows from the just commented property that we can always find a unique element $b\in E_a$ satisfying $b^{[3]} =a$. This element will be denoted by $a^{[\frac13]}$.\smallskip

One of the amazing geometric properties of JB$^*$-triples, known as Kaup's Banach-Stone theorem\label{Kaup Banach-Stone}, asserts that a linear bijection between JB$^*$-triples is an isometry if and only if it is a triple isomorphism, i.e., it preserves triple products (cf. \cite[Proposition 5.5]{Kaup_RiemanMap} or  \cite[Theorem 5.6.57]{CabreraPalaciosBook}). In particular, each JB$^*$-triple admits at unique triple product.

\section{Truncations in \texorpdfstring{JB$^*$-}{}triples}\label{sec: truncations}

As we have commented in the introduction, the notion of truncation defined for elements in $B(H)$ also makes sense in the wider setting of JB$^*$-triples. The corresponding definition we introduce here for JB$^*$-triples is a literal translation. 

\begin{defn}\label{def truncation} Let $a,b$ be two elements in a JB$^*$-triple $E$. We say that \emph{$a$ is a truncation of $b$} if $\{a,a,a\} = \{a,b,a\}$.
\end{defn}

Observe that on $B(H)$ the above definition agrees with the usual notion of truncation.\smallskip

Let $S$ ba a subset of a JB$^*$-triple $E$. According to the standard notation (see \cite{AyuArzi2016Rickart, PeRu2014}), we define the \emph{(outer) quadratic annihilator} of $S$ by \[S^{\perp_q}=\{a\in E: Q(a) (S)=\{0\}\},\]
and the \emph{inner quadratic annihilator} is defined as \[^{\perp_q}S=\{a\in E: Q(s)(a)=0, \ \forall\, s\in S\}.\] Let $e$ be a tripotent in $E$. Since the Peirce-2 projection $P_2(e)$ coincides with $Q(e)^2$, and $Q(e)$ is an algebra involution on the JB$^*$-algebra $E_2(e)$, it follows that \begin{equation}\label{quadratic annihilator of a tripotent} ^{\perp_q}\{e\} = E_0(e) \oplus E_1(e).
\end{equation} It is interesting to determine the quadratic annihilator of a single element. 

\begin{lem}\label{l quadratic annihilator of an element} Let $a$ be an element in a JB$^*$-triple $E$. Then $$ ^{\perp_q}\{a\} = E\cap \Big( E_0^{**}(r(a)) \oplus E_1^{**}(r(a))\Big) = \{ x\in E: P_2(r(a)) (x) =0\},$$ where $r(a)$ denotes the range tripotent of $a$ in $E^{**}$. Furthermore, if $E$ is a JBW$^*$-triple, we have $$ ^{\perp_q}\{a\} =  E_0(r(a)) \oplus E_1(r(a)) = \{ x\in E: P_2(r(a)) (x) =0\},$$ where $r(a)$ denotes the range tripotent of $a$ in $E.$
\end{lem}

\begin{proof} Let us begin with an observation. Suppose $a$ and $z$ are two elements in a C$^*$-algebra $A$ such that $a z =0$ (respectively, $z a =0$), it is well known, and easily deduced from functional calculus, that $r(a^*a) z =0$ (respectively, $z r(aa^*)=0$), where $r(a^*a)$ (respectively, $r(aa^*)$) denotes the right (respectively, left) range projection of $a$ in $A^{**}$. In particular, $a z a =0$ implies that $ r(a^*a)  z r(aa^*)=0$. It is also part of the folklore of the theory that the range tripotent of $a$ in $A^{**}$, $r(a),$ is a partial isometry in $A^{**}$ satisfying $r(a) r(a)^* = r(aa^*)$ and $r(a)^* r(a) = r(a^*a)$. Therefore, $P_2 (r(a))  (z^*) = r(a) r(a)^* z^* r(a)^* r(a) =0$. \smallskip
	
Take now a JB$^*$-algebra $\mathfrak{A}$ and two elements $a,z\in \mathfrak{A}$ with $a$ self-adjoint such that $U_{a} (z)=0$. If we write $z = h + i k$ with $h$ and $k$ self-adjoint, we arrive to the conclusion that $U_{a} (h) = U_{a} (k)=0$. By the Shirshov-Cohn theorem (see also \cite[Corollary 2.2]{Wright1977}), the JB$^*$-subalgebra, $\mathfrak{B},$ of $\mathfrak{A}$ generated by $a$ and $h$ is a JB$^*$-subalgebra of some C$^*$-algebra $A$. The condition $0 = U_a (h) = a h a$ in $A$, and the conclusion in the first paragraph gives $P_2 (r(a)) (h) =0$ (in $A$, and hence in $\mathfrak{B}$ and in $\mathfrak{A}$). Similarly, $P_2 (r(a)) (k) =0,$ and thus $P_2 (r(a)) (z) =0$.\smallskip
 	
Finally, take an arbitrary $a\in E$ and consider everything inside the JBW$^*$-triple $E^{**}$. Suppose that $z\in E$ satisfies $Q(a) (z) =0$. Note $r = r(a)\in E^{**}$ and decompose $z = P_2(r) (z) + P_1(r) (z) +P_0(r) (z) = z_2 + z_1 + z_0$ in $E^{**}$.  By Peirce arithmetic, $0 =\{a,z,a\} = \{a,z_2,a\} = U_{a} (z^{*_{{r(a)}}}_2 )$. Since $a$ is positive in $E^{**}_2(r(a))$, the conclusion in the previous paragraph proves that $P_2(r(a)) (z^{*_{{r(a)}}}_2 ) =0 $, equivalently, $z_2 = P_2(r(a)) (z)=0$. This proves that $ ^{\perp_q}\{a\} \subseteq E\cap \Big( E_0^{**}(r(a)) \oplus E_1^{**}(r(a))\Big)$. The equality is a straightforward consequence of Peirce arithmetic.\smallskip

The final statement is clear from the above discussion.     	
\end{proof}

\begin{rem}\label{rem qannihilator for an element in the Peirce-2} Let $e$ be a tripotent in a JB$^*$-triple $E$. Suppose $a$ is an element in $E_2(e)$. Then ${^{\perp_q}\{e\}}\subseteq {^{\perp_q}\{a\}}$. Namely, if $x\in {^{\perp_q}\{e\}}$, it follows that $x\in E_0(e) \oplus E_1 (e)$, and hence by Peirce arithmetic, $\{a,x,a\} =0$, which proves that $x\in {^{\perp_q}\{a\}}.$
\end{rem}

As in the case of operators, the relation ``being a truncation of'' admits several reformulations in terms of range tripotents and quadratic annihilators.

\begin{lem}\label{lem charact truncation}
Let $E$ be a JB$^*$-triple. For any $a,b \in E$, the following are equivalent:
\begin{enumerate}[{$(a)$}]
    \item $a$ is a truncation of $b$;
    \item $b=a+z$ for some $z\in E$ with $\{a,z,a\}=0$ {\rm(}i.e. $z\in {^{\perp_q}\{a\}} ${\rm)};
    \item $a=P_2(r(a)) (b)$ where $r(a)$ is the range tripotent of $a$ in $E^{**}$. 
\end{enumerate}
\end{lem}

\begin{proof} $(a) \Rightarrow  (b)$ If $a$ is a truncation of $b$, $\{a,a,a\} = \{a,b,a\}$, or equivalently, $\{a, b-a, a\}=0$. That is, $b-a \in {^{\perp_q}\{a\}}.$ Lemma~\ref{l quadratic annihilator of an element} implies that $P_2 (r(a)) (b-a) =0$. Set $z = b-a\in E$ to get the desired element in the statement. \smallskip

$(b) \Rightarrow  (c)$ Suppose $b = a +z$ with $\{a,z,a\} =0$. Lemma~\ref{l quadratic annihilator of an element} assures that $P_2 (r(a)) (z) =0$, and hence $$P_2(r(a))(b)=P_2(r(a))(a+z)=P_2(r(a))(a)+P_2(r(a)) (z)=P_2(r(a))(a)=a.$$

$(c) \Rightarrow  (a)$ By setting $r=r(a),$ we have
$$a^{[3]}=P_2(r)(\{a,a,a\})= P_2(r)(\{a,P_2(r)(b),a\})=\{P_2(r)a,b,P_2(r)a\}=\{a,b,a\},$$ where in the second equality we applied Peirce arithmetic.
\end{proof}

An interesting consequence of Lemma~\ref{lem charact truncation} shows that, as in the case of projections in $B(H)$, when restricted to tripotents the relation ``being a truncation of'' is an equivalent reformulation of the partial ordering among these elements.  

\begin{cor}\label{rem trunc for tripotents}
Let $v,e$ be two tripotents in a JB$^*$-triple $E$. Then the following statements are equivalent:
\begin{enumerate}[$(a)$]
        \item $e$ is a truncation of $v$;
        \item $e\leq v$.
    \end{enumerate}
\end{cor}

\begin{proof} $(a)\Rightarrow (b)$ By Lemma~\ref{lem charact truncation} $e$ is a truncation of $v$ if and only if $v = e + z$ for some $z\in {^{\perp_q}\{e\}} = E_0(e) \oplus E_1(e)$ (cf. \eqref{quadratic annihilator of a tripotent}). In particular,  $P_2(e) (v) = e$, which by \cite[Corollary 1.7]{FriedRusso_Predual} is equivalent to $e\leq v$. \smallskip
	
The implication $(b)\Rightarrow (a)$ is clear since $e\leq v$ if and only if $v-e$ is a tripotent orthogonal to $e$. In particular $\{e,v,e\} = \{e,e,e\} + \{e,v-e,e\} = \{e,e,e\}=e$.	
\end{proof}

\section{Preservers of the truncation of triple products}\label{sec: preservers of the trunc of triple products}

Once we have introduced the notion of truncation in the setting of JB$^*$-triples, it becomes evident that the result by Jia, Shi and Ji commented in the introduction (see \cite{Jia_Shi_Ji_AnnFunctAnn_2022}) motivates the following definition and problem on preservers. 

\begin{defn}\label{def preservers of truncations of triple products} Let $E, F$ be two JB$^*$-triples and let $\Delta:E\to F$ be a {\rm(}non-necessarily linear{\rm)} bijective  mapping. We said that $\Delta$ \textit{preserves the truncation of triple products} if 
\begin{align}\label{define trunc preservers one way}
		\boxed{a \mbox{ is a truncation of } Q(b)(c)} \Rightarrow \boxed{\Delta(a)  \mbox{ is a truncation of } Q(\Delta(b)) (\Delta(c))}.  \end{align} 
If the equivalence \begin{align}\label{define trunc preservers}
	\boxed{a \mbox{ is a truncation of } Q(b)(c)} \Leftrightarrow \boxed{\Delta(a)  \mbox{ is a truncation of } Q(\Delta(b)) (\Delta(c))}\end{align} holds we say that $\Delta$ \textit{preserves the truncation of triple products in both directions}.  
\end{defn}

It is clear that \cref{define trunc preservers} is equivalent to:
\begin{align}\label{char trunc preservers}
	\boxed{a^{[3]}=\{a,Q(b)(c),a\}}  \Leftrightarrow \boxed{\Delta(a)^{[3]}=\{ \Delta(a),  Q(\Delta(b)) (\Delta(c)),\Delta(a)\}}.  \end{align} By replacing ``$\Leftrightarrow$'' with ``$\Rightarrow$'' in the previous identity we get a reformulation of \eqref{define trunc preservers one way}.

\begin{probl}\label{prob1} Let $\Delta: E\to F$ be a {\rm(}non-necessarily linear{\rm)} bijection between JB$^*$-triples. Suppose that $\Delta$ preserves the truncation of triple products in both directions. Is $\Delta$ a (continuous, actually isometric) linear triple isomorphism? 
\end{probl}

The main result in this paper gives a positive answer to the above problem in the case that $E$ and $F$ are two Cartan  factors of rank greater than or equal to $2$, or more generally, two atomic JBW$^*$-triples not containing Cartan factors of rank-one. If we also assume that our mapping $\Delta$ is continuous when restricted to rank-one Cartan factors, we also arrive to the same conclusion.   \smallskip

In this section we study the properties of  bijective preservers of truncations of triple products in both directions in the most general setting assuming that domain and codomain are general JB$^*$-triples. For this purpose, along this section, unless otherwise indicated, $E$ and $F$ will stand for two JB$^*$-triples, while $\Delta : E \to F$ will be a {\rm(}non-necessarily linear{\rm)} bijective map preserving the truncation of triple products in both directions (see \cref{define trunc preservers}, \cref{char trunc preservers}). 

\begin{lem}\label{lem preserve annihil} The following statements hold:
\begin{enumerate}[$(i)$]
\item\label{lem preserve annihil 0} $\Delta(0)=0;$	
\item\label{lem preserve annihil 1} For all $a,b\in E$ we have $Q(a)(b)=0$ if and only if $ Q(\Delta(a))(\Delta(b))=0;$
\item\label{lem preserve annihil 2} $\Delta(S^{\perp_q})=\Delta(S)^{\perp_q},$ for every subset $S\subseteq E$;
\item\label{lem preserve annihil 3} $\Delta({^{\perp_q}S})={^{\perp_q}\Delta(S)}$, for every subset $S\subseteq E$;
\item\label{lem preserve annihil 2 orth to ann} If $a,b\in E$ with $a\perp b,$ we have $Q(\Delta(a))(\Delta(b))=Q(\Delta(b))(\Delta(a))=0.$
\end{enumerate}
\end{lem}

\begin{proof}(\!\!\cref{lem preserve annihil 0}) It is clear that $0$ is a truncation of $b^{[3]}=Q(b)b,$ for every $b\in E.$ Thus, by hypotheses we have $\Delta(0)^{[3]}=\{ \Delta(0),\{\Delta(b),\Delta(b),\Delta(b) \},\Delta(0)\} $ for every $b\in E.$ By surjectivity there exists $b\in E$ such that $\Delta(b)=0,$ whence $\Delta(0)=0,$ by the Gelfand--Naimark axiom in the definition of JB$^*$-triple.\smallskip

(\!\!\cref{lem preserve annihil 1}) Assume first that $Q(a) (b)=0$. Set $d=Q(\Delta(a))(\Delta(b)).$ Then $d^{[3]}=\{d,Q(\Delta(a))(\Delta(b)),d\} .$ Since $\Delta^{-1}$ also satisfies \cref{char trunc preservers} we get
$$\Delta^{-1}(d)^{[3]}=\{\Delta^{-1}(d),Q(a)(b),\Delta^{-1}(d)\}=0,$$ which implies that $ \Delta^{-1}(d)=0,$ and thus $Q(\Delta(a))(\Delta(b))=d=0.$ For the other implication, we just note that $\Delta$ also satisfies \cref{char trunc preservers}, and hence the previous argument applied to $\Delta^{-1}$ gives the implication.\smallskip

The statements (\!\!\cref{lem preserve annihil 2})
and (\!\!\cref{lem preserve annihil 3}) are clear consequences of (\!\!\cref{lem preserve annihil 1}).\smallskip
 
(\!\!\cref{lem preserve annihil 2 orth to ann}) If $a\perp b$ we have $\{a,b,a\}=\{b,a,b\}=0.$ Now, by applying   (\!\!\cref{lem preserve annihil 1}) we get $Q(\Delta(a))(\Delta(b))=Q(\Delta(b))(\Delta(a))=0$.
 \end{proof}
 
In our next result we reveal that bijective preservers of truncations of triple products in both directions preserve von Neumann regularity in both directions.\smallskip  

Recall that an element $a$ in a JB$^*$-triple $E$ is called \textit{von Neumann regular} if there exists $b\in E$ such that $Q(a)(b)=a.$ This element $b$ is not, in general, unique. However, if $a$ is von Neumann regular, there exists a unique $b$ in $E$ satisfying $Q(a)(b)=a,$ $Q(b)(a)=b$ and $[Q(a),Q(b)]:=Q(a)Q(b)-Q(b)Q(a)=0$ (see \cite[Lemma 4.1]{Ka96} or \cite{BurKaMoPeRa,FerGarSanSi92,FerGarSanSi94}). The unique element $b\in E$ satisfying the previous properties is called the \textit{generalized inverse} of $a$ in $E$, and will be denoted by $a^{\dag}$. The symbol $E^{\dag}$ will stand for the set of all von Neumann regular elements in $E$.

\begin{lem}\label{lem pres vn regular}
    $\Delta$ preserves von Neumann regularity in both directions, that is, an element $a$ in $E$ is von Neumann regular if and only if $\Delta (a)$ is. Furthermore if $a$ is von Neumann regular, $\{\Delta (a^{\dag}), \Delta(a),\Delta (a^{\dag})\}  = \Delta (a^{\dag})$ and $\{\Delta(a),\Delta (a^{\dag}),\Delta(a) \} = \Delta(a)$. 
\end{lem}

\begin{proof} Let $a\in E$ be a von Neumann regular element with generalized inverse $a^{\dag}$. Since $a^{[3]}=\{a,Q(a)(a^{\dag}),a\},$ by hypotheses we have  
$$\Delta(a)^{[3]}=\{ \Delta(a),Q( \Delta(a))(\Delta(a^{\dag})) ,\Delta(a)\}.$$ We deduce form the above equality that  $Q(\Delta(a))\Big( \Delta(a)-Q( \Delta(a)) (\Delta(a^{\dag})) \Big)=0.$ Lemma~\ref{l quadratic annihilator of an element} implies that $P_2 (r(\Delta(a))) \Big( \Delta(a)-Q(\Delta(a)) (\Delta(a^{\dag})) \Big) = 0.$ However, since $\Delta(a)$ and $Q(\Delta(a)) (\Delta(a^{\dag}))$ are elements in the unital JBW$^*$-algebra $F_2^{**}(r(\Delta(a)))$ (just apply Peirce arithmetic), we get  $\Delta(a)=Q( \Delta(a)) (\Delta(a^{\dag}))$. Since $a$ is the generalized inverse of $a^{\dag},$ it follows from the previous conclusion that  $\Delta(a^\dag)=Q( \Delta(a^\dag)) (\Delta(a))$. This proves that $\Delta (a)$ is von Neumann regular. The same conclusion applied to $\Delta^{-1}$ shows that $a$ is von Neumann regular whenever $\Delta (a)$ is.\end{proof}

As a corollary we can now prove that bijective preservers of truncations of triple products in both directions also preserve tripotents in both directions.

\begin{lem}\label{lem pres tripotents and leq} $\Delta$ preserves tripotents and partial order among tripotents in both directions. In particular $\Delta$ maps bijectively (order) minimal and maximal tripotents to (order) minimal and maximal tripotents, respectively, that is, $\Delta\left( \mathcal{U}_{min}(E) \right) = \mathcal{U}_{min}(F)$, and $\Delta\left( \mathcal{U}_{max}(E) \right) = \mathcal{U}_{max}(E)$.  
\end{lem}

\begin{proof} Let $e$ be a tripotent in $E$. It is clear, and well-known, that $e$ is von Neumann regular with $e^{\dag} = e$. It follows from the previous Lemma~\ref{lem pres vn regular} that $\Delta(e)$ is von Neumann regular and $\{\Delta(e), \Delta(e) , \Delta(e)\} = \Delta (e),$ which proves that $\Delta (e)$ is a tripotent. This conclusion applied to $\Delta^{-1}$ assures that $\Delta$ preserves tripotents in both directions. The final statement is a consequence of the hypothesis on $\Delta$,  Corollary~\ref{rem trunc for tripotents} and the fact that $e\leq v$ in $\mathcal{U} (E)$ if and only if $e$ is a truncation of $v =\{v,v,v\}$. 
\end{proof}

\begin{rem}\label{r Delta preserves the truncation of tripotents} It should be remarked that any mapping $\Delta$ satisfying our hypotheses preserves the truncation of tripotents. Namely, suppose $a$ is a truncation of a tripotent $e$ in the domain JB$^*$-triple $E$, that is, $\{a,a,a\}= \{a,e,a\}$. Since $\{a,a,a\} = \{a,e,a\} = \{a,\{e,e,e\},a\}= \{a,Q(e)(e),a\},$ the properties of $\Delta$ and \cref{lem pres tripotents and leq} assure that $\Delta(a)$ is a truncation of $\Delta(e)$. 
\end{rem}

Now, as a consequence of \cref{lem preserve annihil}$(iv)$ and the fact that $\Delta\left({^{\perp_q}\{a\}} \right) = {^{\perp_q}\{\Delta(a)\}}$, \cref{lem pres tripotents and leq} and \eqref{quadratic annihilator of a tripotent} we obtain:

\begin{cor}\label{c preservation Perice 1 + 0} The equality  \begin{align}\label{eq preser Peirce 1+0 for e}
		\Delta(E_1(e)\oplus E_0(e))=F_1(\Delta(e))\oplus F_0(\Delta(e))
	\end{align} holds for every tripotent $e$ in $E$. \hfill$\Box$
\end{cor} 

Our next corollary will also follow from Lemma~\ref{l quadratic annihilator of an element} and Lemma~\ref{lem preserve annihil}$(iv)$.

\begin{cor}\label{c preservation Perice 1 + 0 for an element} The equality  \begin{align}\label{eq preser Peirce 1+0 for a} \Delta\left(E\cap \Big( E_1^{**}(r(a))\oplus E_0^{**}(r(a))\Big) \right)=F\cap \Big( F_1^{**} (r(\Delta(a))) \oplus F_0^{**}(r(\Delta(a)))\Big)
	\end{align} holds for every element $a$ in $E$. Furthermore, if $E$ and $F$ are JBW$^*$-triples, we also have 
\begin{align}\label{eq preser Peirce 1+0 for a JBW}
	\Delta\Big(E_1 (r(a))\oplus E_0(r(a))\Big)=F_1(r(\Delta(a))\oplus F_0(r(\Delta(a))).
\end{align}	
	\hfill$\Box$
\end{cor}

We are now in a position to prove certain kind of additivity of our mapping $\Delta$ on orthogonal elements. 
    
\begin{lem}\label{lem almost orth add} Suppose $a$ and $b$ are two orthogonal elements in $E$. Then there exists an element $z\in {^{\perp_q}\{\Delta(a)}\}\cap {^{\perp_q}\{\Delta(b)\}}$ such that $$\Delta\left(a^{[\frac{1}{3}]}+b^{[\frac{1}{3}]}\right)^{[3]}=\Delta(a)+\Delta(b)+z.$$ 
   \end{lem}
   
\begin{proof}
Since $a\perp b,$ it easily follows that $a^{[\frac{1}{3}]}\perp b^{[\frac{1}{3}]}$ (cf. \cite[Lemma 1.1]{BFPGMP08}), and hence  
\begin{equation}\label{eq 1 lem almost orth add}
a^{[3]}=\{a,(a^{[\frac{1}{3}]}+b^{[\frac{1}{3}]})^{[3]},a\}, \ \ \hbox{ and } \ \  b^{[3]}=\{b,(a^{[\frac{1}{3}]}+b^{[\frac{1}{3}]})^{[3]},b\}
\end{equation}
which, by the hypotheses on $\Delta$, implies that  
\begin{equation}\label{eq 2 lem almost orth add}\begin{aligned}
		&\Delta(a)^{[3]}=\{\Delta(a),\Delta(a^{[\frac{1}{3}]}+b^{[\frac{1}{3}]})^{[3]},\Delta(a)\},\\ &\hbox{and } \\ &\Delta(b)^{[3]}=\{\Delta(b),\Delta(a^{[\frac{1}{3}]}+b^{[\frac{1}{3}]})^{[3]},\Delta(b)\}.
	\end{aligned}
\end{equation}
    
By \cref{lem charact truncation} there exists $z_a\in {^{\perp_q}\{\Delta(a)\}},z_b\in {^{\perp_q}\{\Delta(b)\}}$ such that 
\begin{align}\label{eq 3 lem almost orth add}
\Delta(a^{[\frac{1}{3}]}+b^{[\frac{1}{3}]})^{[3]}=\Delta(a)+z_a \quad\mbox{and} \quad\Delta(a^{[\frac{1}{3}]}+b^{[\frac{1}{3}]})^{[3]}=\Delta(b)+z_b
\end{align} hold. Since by \cref{lem preserve annihil}$(iv)$, $\Delta(b)\in {^{\perp_q}\{\Delta(a)\}},$ by combining the second equality in \cref{eq 3 lem almost orth add} with the first one in \cref{eq 2 lem almost orth add} we have 
$$ \Delta(a)^{[3]} = \{ \Delta(a),\Delta(b) +z_b,\Delta(a)\} =\{ \Delta(a),z_b,\Delta(a)\},$$ and hence, by a new application of Lemma~\ref{lem charact truncation}, there exists $z\in {^{\perp_q}\{\Delta(a)\}}$ such that $z_b=\Delta(a)+z.$ We therefore have $\Delta(a^{[\frac{1}{3}]}+b^{[\frac{1}{3}]})^{[3]}=\Delta(a)+\Delta(b)+z.$ Now we plug this identity in the second equality of \cref{eq 2 lem almost orth add}, combined with the fact that $\Delta(a)\in {^{\perp_q}\{\Delta(b)\}}$ (cf. \cref{lem preserve annihil}$(iv)$) to show that 
$$\Delta(b)^{[3]}=\{\Delta(b),\Delta(a)+\Delta(b)+z,\Delta(b)\} =\{\Delta(b),\Delta(b)+z,\Delta(b)\},$$
which by virtue of \cref{lem charact truncation}, implies that $z\in {^{\perp_q}\{\Delta(b)\}}.$
\end{proof}

\begin{lem}\label{lem orthogonal tripotents}
Suppose $e_1,\dots,e_n$ are mutually orthogonal tripotents in $E$. Then $\Delta(e_1),\dots,\Delta(e_n)$ are mutually orthogonal tripotents in $F$ and $$\Delta(e_1+\dots+e_n)=\Delta(e_1)+\dots+\Delta(e_n).$$
\end{lem}

\begin{proof}
It is enough to show the result for $n=2.$ The statement for $n\in \mathbb{N}$ will  then easily follow by induction if we observe that the sum of a finite family of mutually orthogonal tripotents is a tripotent.\smallskip

Let $e_1$ and $e_2$ be two orthogonal tripotents in $E$.  By \cref{lem pres tripotents and leq} $\Delta(e_1),\Delta(e_2)$ and $\Delta(e_1+e_2)$ are tripotents in $F$ and $\Delta$ preserves the partial order among tripotents in both directions.\smallskip 

Back to the arguments in the proof of \cref{lem almost orth add} with $a= e_1$ and $b = e_2$, $\Delta(e_1)$ and $\Delta(e_2)$ are tripotents, we see in \eqref{eq 2 lem almost orth add} that $\Delta(e_1)$ and $\Delta (e_2)$ are truncations of $\Delta (e_1 + e_2)$, and thus \cref{rem trunc for tripotents} affirms that $z_a$ and $z_b$ in the commented proof are tripotents in $F$ with $z_a\perp \Delta(e_1)$ and $z_b \perp \Delta(e_2)$. Following the same arguments we obtain that the tripotent $\Delta (e_1)$ is a truncation of the tripotent $z_b$, and hence, again by \cref{rem trunc for tripotents}, $z_b = \Delta (e_1) + z$, where $z$ is a tripotent in $F$ with $z\perp \Delta (e_1)$ and $\Delta (e_1 + e_2) = \Delta (e_1) + \Delta (e_2) + z$. Note that $\Delta (e_1), z\leq \Delta (e_1) + z = z_b\perp \Delta (e_2)$, and consequently $\Delta (e_1), z\perp \Delta(e_2)$.\smallskip   

Observe that $\Delta^{-1}$ also preserves tripotents, order, and orthogonality among them. Thus, $\Delta^{-1}(z)$ is a tripotent, $e_1=\Delta^{-1}( \Delta(e_1)) \perp \Delta^{-1}(z), e_2=\Delta^{-1}( \Delta(e_2)) \perp \Delta^{-1}(z),$ and the inequality $e_1,e_2,\Delta^{-1}(z)\leq e_1+e_2$ holds. However, in such a case  $\Delta^{-1}(z)\perp e_1+e_2$ and $\Delta^{-1}(z)\in E_2(e_1+e_2).$ This shows that necessarily $\Delta^{-1}(z)=0,$ and hence $z=0.$
\end{proof}

The direct sum of the Peirce-1 and -0 subspaces associated with a tripotent $e$ in $E$ is clearly preserved by $\Delta$ (cf. Corollary~\ref{c preservation Perice 1 + 0 for an element}). We deal next with the Peirce-2 subspace.  

\begin{prop}\label{Peirce2toPeirce2}
    Let $e\in E$ be a tripotent in $E$. Then $\Delta(E_2(e)) = F_2(\Delta(e)).$
\end{prop}

\begin{proof} Fix $a\in E_2(e).$ Then $^{\perp_q}\{e\}\subseteq ^{\perp_q}\{a\}$ (cf. Remark~\ref{rem qannihilator for an element in the Peirce-2}). Thus, by statement (\!\!\cref{lem preserve annihil 3}) in Lemma~\ref{lem preserve annihil},
    \begin{align}\label{eq prop pres peirce 2}
       F_0(\Delta(e))\oplus F_1(\Delta(e))={^{\perp_q}\{\Delta(e)\}}\subseteq ^{\perp_q}\{\Delta(a)\}.
    \end{align}
 Set $x_i=P_i(\Delta(e)) (\Delta(a)),$ in particular $\Delta(a) = x_0+x_1+x_2$. It follows from \cref{eq prop pres peirce 2} that $x_0,x_1\in {^{\perp_q}\{\Delta(a)\}}.$ Therefore,
 $$\begin{aligned}
 	0=\{\Delta(a),x_0,\Delta(a)\}&=\{x_2+x_1+x_0,x_0,x_2+x_1+x_0\}=(x_0\perp x_2)  \\
 	&=\{x_1+x_0,x_0,x_1+x_0\}=\{x_1,x_0,x_1\}+x_0^{[3]}+2\{x_1,x_0,x_0\}.
 \end{aligned}$$ 
 Since, by Peirce arithmetic,  $\{x_1,x_0,x_1\}\in F_2(\Delta(e))$, $\{x_0,x_0,x_0\}\in F_0(\Delta(e))$ and $\{x_1,x_0,x_0\} \in F_1(\Delta(e))$, we deduce that $x_0^{[3]}=\{x_1,x_0,x_1\}=\{x_1,x_0,x_0\}=0,$ and hence $x_0 =0$ and $\Delta(a)=x_2+x_1.$\smallskip 
 
Note now that $x_1 \in {^{\perp_q}\{\Delta(a)\}}$, so $$0=\{x_2+x_1,x_1,x_2+x_1\}=\{x_2,x_1,x_2\}+x_1^{[3]}+2\{x_2,x_1,x_1\}=x_1^{[3]}+2\{x_2,x_1,x_1\}.$$ It follows from $x_1^{[3]}\in F_1(\Delta(e))$ and $\{x_2,x_1,x_1\} \in F_2(\Delta(e))$ that $x_1^{[3]}=\{x_2,x_1,x_1\}=0.$ Thus, $x_1=0$ and $\Delta(a)=x_2\in F_2(\Delta(e)).$\smallskip

We have shown that $\Delta(E_2(e)) \subseteq F_2(\Delta(e)),$ for any mapping $\Delta$ under our hypotheses. Since $\Delta^{-1}$ satisfies the same property, we arrive to 
$\Delta^{-1} \left(F_2(\Delta(e))\right) \subseteq E_2(e),$ which concludes the proof.
\end{proof}

As a consequence of \cref{Peirce2toPeirce2} we obtain that for each minimal tripotent $e$ in $E$ we have
$$\Delta(\CC e) = \Delta(E_{2}(e)) = F_{2}(\Delta(e)) = \CC \Delta(e).$$
Thus, there exists a bijection $f_{e} : \CC \to \CC,$\label{def of fe} which depends on $e,$ satisfying $\Delta (\lambda e )= f_{e}(\lambda) \Delta(e)$ for any $\lambda \in \CC$. \smallskip

We shall next study the properties of the mappings $f_e$.

\begin{cor}\label{cor of fmultiplicative}
Let $e$ be a minimal tripotent in $E$. The following properties hold for all $\lambda, \mu \in \CC \backslash \{0\}$:
\begin{enumerate}[$(i)$]\item\label{c 311 f(0)=0} $f_e (0) =0$ and $f_e (\mathbb{T}) = \mathbb{T}$, where $\mathbb{T}$ denotes the unit sphere of $\mathbb{C}$; 
\item\label{fpreserves vn regular emelemnts} $\Delta (Q(a)(b))= Q(\Delta(a)) (\Delta(b)),$ for all $a, b \in E_{2}(e)$;
\item\label{fpreserves triple multiplicative} $ f_{e}(\lambda^2 \Bar{\mu}) = {f_{e}(\lambda)}^{2}\  \overline{f_{e}(\mu) }$;
    \item\label{fsquare and conjugate} ${f_{e}(\lambda)}^2 = f_{e}(\lambda ^2)$ and $f_{e}(\Bar{\lambda}) = \overline{f_{e}(\lambda)}$;
    \item\label{fmultiplicative} $f_{e}(\lambda \mu)=f_{e}(\lambda) f_{e}(\mu)$. 
\end{enumerate}
\end{cor}

\begin{proof}$(\!\!\!\cref{c 311 f(0)=0})$ The first conclusion is a consequence of the fact that $\Delta(0)=0$ (cf. \cref{lem preserve annihil}$(i)$), while the second one follows from \cref{lem pres tripotents and leq}, the definition of $f_e$ (or \cref{Peirce2toPeirce2}) and the minimality of $e$ and $\Delta (e)$.\smallskip

$(\!\!\!\cref{fpreserves vn regular emelemnts})$  
Fix arbitrary $a, b \in E_{2}(e) = \mathbb{C} e$. If $a=0$ or $b=0$, it follows from \cref{lem preserve annihil}$(i)$ that $\Delta(a) = 0$ or $\Delta (b)=0$, and hence the desired equality is trivial. We can therefore assume that $a,b\neq 0$, equivalently, $a = \lambda e$ and $b = \mu e$ for some $\lambda,\mu\neq 0$. It follows from the fact that $Q(a)b$ is a truncation of itself that 
$$Q(\Delta (Q(a)(b))) (\Delta(Q(a)(b))- Q(\Delta(a))(\Delta(b)) ) = 0,$$
that is, $\Delta (Q(a)(b)) = Q(\Delta(a))(\Delta(b))$ since clearly $\Delta (Q(a)(b)) = \Delta (\lambda^2\overline{\mu} e ) = f_e (\lambda^2\overline{\mu}) Q(\Delta(e))$ is an invertible mapping on $F_2(\Delta(e))$.\smallskip

$(\!\!\!\cref{fpreserves triple multiplicative})$ Observe that  $\{\lambda e, \mu e, \lambda e\}$ is in $E_{2}(e)$, then by $(\!\!\!\cref{fpreserves vn regular emelemnts})$ we have  
$$\begin{aligned}
	f_{e}(\lambda^2 \Bar{\mu}) \Delta(e) &= \Delta (\lambda^2 \ \bar{\mu}e) = \Delta(\{\lambda e, \mu e, \lambda e \}) = \{\Delta(\lambda e),\Delta(\mu e), \Delta(\lambda e) \} \\
	&= \{f_{e}(\lambda) \Delta(e),f_{e}(\mu) \Delta( e), f_{e}(\lambda) \Delta( e) \}  = {f_{e}(\lambda)}^{2} \overline{f_{e}(\mu) }  \Delta (e). 
\end{aligned}$$

$(\!\!\!\cref{fsquare and conjugate})$ It follows immediately from  $(\!\!\!\cref{fpreserves triple multiplicative})$.\smallskip

$(\!\!\!\cref{fmultiplicative})$ This statement is evident from the previous one, since every element $\lambda \in \CC$ can be wrote as a square of some other complex number.
\end{proof}

\begin{rem}\label{r prop in cor 3.11 are not enough}  The mapping $f : \mathbb{C}\to \mathbb{C}$ given by $f(\lambda) = \lambda^{-1}$ if $\lambda\neq 0$ and $f(0)=0$ satisfies all the conclusions for the map $f_e$ in \cref{cor of fmultiplicative}, but it is not linear. 
\end{rem}

We gather in the next lemma some sufficient conditions to guarantee that our mapping $f_e$ is the identity. We gather some facts in the proof of \cite[Theorems 2.1 and 3.1]{Jia_Shi_Ji_AnnFunctAnn_2022}, needed in our arguments, which perhaps can be available for other purposes. 

\begin{lem}\label{l suffcient cond for fe identity}
Let $f:\mathbb{C}\to \mathbb{C}$ be a bijection satisfying the following hypotheses for all $\lambda, \mu \in \CC \backslash \{0\}$:
\begin{enumerate}[$(i)$]
\item $f(\Bar{\lambda}) = \overline{f(\lambda)}$;
\item $f(\lambda \mu)=f(\lambda) f(\mu)$;
\item $f(t) =t,$ for all $t\in [0,1]$. 
\end{enumerate} Then $f (r) = r$ for all $r\in \mathbb{R}.$ If we further assume that $f(\lambda) = \lambda $ for all unitary $\lambda \in \mathbb{C}$, the mapping $f$ is the identity on $\mathbb{C}$. 
\end{lem}   

\begin{proof} Observe that the hypotheses $(i)$ and $(ii)$ already assure that $f(0)=0$, $f(1) =1,$ $f(\lambda^{-1}) = f(\lambda)^{-1}$ for all $\lambda\neq 0$, and $f(\mathbb{T}) = \mathbb{T}$. Namely, by the surjectivity we can find $\lambda \in \mathbb{C}$ with $f(\lambda) = 0$. Then $f (\lambda \mu ) = f(\lambda) f(\mu) = 0=f(\lambda)$ for all $\mu$. It follows from the injectivity of $f$ that $\lambda \mu = \lambda $ for all $\mu \in \mathbb{C}$, and then $\lambda =0$. The rest is clear. \smallskip
	
By $(ii)$ and $(iii),$ for any $\lambda \in [1,+\infty)$, we have $f(\lambda) = f(\lambda^{-1})^{-1} = \lambda$.  So $f(\lambda) = \lambda $ for all $\lambda \in\mathbb{R}_0^+$.  Since $f(-1)^2 = f(1) = 1$, we can only have $f(-1) = -1$ due to the injectivity of $f$ and $f(1) =1$. Therefore, for $\lambda<0$, $f(\lambda)= f(-1) f(-\lambda) = \lambda.$\smallskip

The final statement can be obtained by just observing that for each non-zero $\lambda$ in $\mathbb{C}$ we have $f(\lambda ) = f(\frac{\lambda}{|\lambda|} \ |\lambda|) =  f(\frac{\lambda}{|\lambda|}) f(|\lambda|) = \lambda.$
\end{proof}

\subsection{Preservers of the truncation of triple products on \texorpdfstring{JBW$^*$-}{}algebras}\ \smallskip

The motivation to study the particular setting in which the domain of our mapping $\Delta$ is a JBW$^*$-algebra comes from \cref{Peirce2toPeirce2} and the fact that for element $a$ in a JBW$^*$-triple $E$ there exists a maximal tripotent $e\in E$ such that $a\in E_2(e)$ (cf. \cite[Lemma 3.12]{Horn_MathScand_1987}). Recall that in this case $E_2(e)$ a JBW$^*$-algebra, that is, a JB$^*$-algebra which is additionally a dual Banach space (it is in particular unital). In this subsection we shall additionally assume that $E$ is a unital JBW$^*$-algebra, and $F$ is a JBW$^*$-triple.\smallskip

According to the standard notation, the self-adjoint idempotents in a JB$^*$-algebra $A$ are called \emph{projections}, and we write $A_{sa}$ for the real subspace of all self-adjoint elements in $A$. Observe that the projections of $A$ are precisely the positive tripotents in $A$. We shall write $\mathcal{P} (A)$ for the (possibly empty) set of all projections in $A$.\smallskip

We recall that an element $a$ in a unital JB$^*$-algebra $A$ is called \emph{invertible} if there exists $b\in A$ satisfying $a \circ b = \mathbf{1}$ and $a^2 \circ b = a.$ The element $b$ is unique, it is called the inverse of $a$ and denoted by $a^{-1}$ (cf. \cite[Definition 4.1.2]{CabreraPalaciosBook}). We know from \cite[Theorem 4.1.3]{CabreraPalaciosBook} that an element $a\in A$ is invertible if and only if $U_a$ is a bijective mapping, and in such a case $U_a^{-1} = U_{a^{-1}}$. Recall that for each element $a$ in a JB$^*$-algebra $A$, the symbol $U_a$ will stand for the operator on $A$ given by $U_a (x) = 2(a\circ x)\circ a - a^2 \circ x$ ($x\in A$). Observe that in this setting $Q(a) (x) =  U_a (x^*)$ for all $x\in A$. Therefore $a$ is invertible if and only if $Q(a)$ is invertible. The symbol $A^{-1}$ will stand for the set of all invertible elements in $A$. An additional property of invertible elements asserts that $U_a (b)\in A^{-1}$ if and only if $a,b\in A^{-1}$ (\cite[Theorem 4.1.3$(vi)$]{CabreraPalaciosBook}). \smallskip

An element $u$ in $A$ is called \emph{unitary} if it is invertible in $A$ with inverse $u^*$. This is not misleading use of the term unitary since unitaries in $A$ in this sense are precisely the unitary tripotents of $A$ when the latter is regarded as a JB$^*$-triple {\rm(}cf. \cite[Proposition 4.3]{BraunKaupUpmeier78}, \cite[Theorem 4.2.24, Definition 4.2.25 and Fact 4.2.26]{CabreraPalaciosBook} or the discussion in \cite[Remark 1.1 and Lemma 2.1]{CUetoPeralta2022LMA}{\rm)}.\smallskip

We already know from previous results is that our mapping $\Delta$ preserves (order) minimal and maximal tripotents. In the case of unital JB$^*$-algebras we also have a special subclass of the maximal tripotents, the unitaries, which are also preserved by $\Delta$.

\begin{lem}\label{lem pres unitaries JB-alg}
Let $A$ be a unital JB$^*$-algebra, let $B$ be a JB$^*$-triple, and let $\widetilde{\Delta} :A\to B$ be a {\rm(}non-necessarily linear{\rm)} bijection preserving the truncation of triple products in both directions. Then $B$ is a unital JB$^*$-algebra and $\widetilde{\Delta}$ preserves unitaries in both directions.
\end{lem}

\begin{proof} Pick a unitary $u\in A.$ Since $u\in \mathcal{U} (A)$, \cref{lem pres tripotents and leq} assures that $\widetilde{\Delta} (u)$ is a tripotent in $B$. \cref{Peirce2toPeirce2} guarantees that $F_2(\widetilde{\Delta}(u))=\widetilde{\Delta}(A_2(u))=\widetilde{\Delta}(A)=B$, where in the last equality we applied that $\widetilde{\Delta}$ is surjective. Therefore $u$ is a unitary in $B$.
\end{proof}

It follows from the previous lemma that, under our hypotheses, the element $u=\Delta(\mathbf{1})$ is a unitary in $F$, and the latter is a (unital) JBW$^*$-algebra with product $a\circ_u b=\{a,u,b\}$ and involution $a^{*_{u}}=\{u,a,u\}.$ By the uniqueness of the triple product in $F$ (cf. page~\pageref{Kaup Banach-Stone} or \cite[Proposition 5.5]{Kaup_RiemanMap}), the mapping $\Delta: E\to F$ is a unital bijection preserving the truncation of triple products in both directions. We therefore assume, without loss of generality, that $\Delta$ is unital. Henceforth we shall write $A$ and $B$ for $E$ and $F$, respectively to emphasize that we are dealing with JBW$^*$-algebras.\smallskip

\begin{prop}\label{prop basic properties Jordan} Let $A,B$ be JB$^*$-algebras, and let $\Delta:A\to B$ be a {\rm(}non-necessarily linear{\rm)} unital bijection preserving the truncation of triple products in both directions. Then the following statements hold: 
\begin{enumerate}[$(a)$]
\item\label{prop basic properties Jordan 1} $\Delta$ preserves projections;
\item\label{prop basic properties Jordan 2}  If $p_1,\dots,p_n$ are mutually orthogonal projections then $$\Delta(p_1+\dots+p_n)=\Delta(p_1)+\dots+\Delta(p_n);$$
\item\label{prop basic properties Jordan 3}  $\Delta$ preserves orthogonality and the partial order of $\mathcal{P}(A)$ (in both directions);
\item \label{prop basic properties Jordan 4} $\{a,a,a\}=\{a,b,a\}$ if and only if $\Delta(a)^{[3]}= \{ \Delta(a),\Delta(b^*)^*,\Delta(a)\};$     
\item \label{prop basic properties Jordan 5} $\Delta(A_{sa}^{-1})\subseteq B_{sa}$, $\Delta^{-1} (B_{sa}^{-1}) \subseteq A_{sa}$;

\item\label{prop basic properties Jordan 6} $\Delta $ preserves the truncation of invertible self-adjoint elements in both directions;

\item\label{new f-g 3.13} For all $a,b\in A_{sa}^{-1}$, we have $\Delta \{a,b,a\} = \{\Delta(a),\Delta(b),\Delta(a)\}$.
\end{enumerate}
\end{prop}

\begin{proof}
$(\!\!\!\cref{prop basic properties Jordan 1}) $ Fix $p\in \mathcal{P}(A)\subset \mathcal{U} (A).$ Since $p\leq \mathbf{1}$, \cref{lem pres tripotents and leq} proves that $\Delta(p)$ is a tripotent in $B$ and $\Delta(p)\leq \mathbf{1} = \Delta(\mathbf{1}).$ Thus, $\Delta(p)$ is a projection in $B$.\smallskip

$(\!\!\!\cref{prop basic properties Jordan 2})$ and $(\!\!\!\cref{prop basic properties Jordan 3})$ are clear from \cref{lem orthogonal tripotents}, \cref{lem pres tripotents and leq} and the fact that the partial ordering on $\mathcal{U}(A)$ extends the partial ordering among projections.\smallskip

$(\!\!\!\cref{prop basic properties Jordan 4})$ Suppose that  $\{a,a,a\}=\{a,b,a\}$. By observing that $a$ is a truncation of $Q(\mathbf{1}) (b^*)$ and $\Delta (\mathbf{1}) = \mathbf{1}$, we get $$\{\Delta(a),\Delta(a),\Delta(a)\} = \{\Delta(a),\{\Delta(\mathbf{1}),\Delta(b),\Delta(\mathbf{1})\}, \Delta (a)\} = \{ \Delta(a),\Delta(b^*)^*,\Delta(a)\}$$ The reciprocal implication follows similarly.\smallskip

$(\!\!\!\cref{prop basic properties Jordan 5})$ Take $a=a^*$ invertible (equivalently, $Q(a)$ invertible). It follows from the invertibility of $Q(a)$ and \cref{lem preserve annihil}$(ii)$ that $Q(\Delta(a))$ is injective. The identity $a^{[3]}=\{a,\{\mathbf{1},a,\mathbf{1}\},a\}$ implies that  $Q(\Delta(a))(\Delta(a))=\Delta(a)^{[3]}=Q(\Delta(a))(\Delta(a)^*).$ Thus $\Delta(a)=\Delta(a)^*.$ Similar arguments prove $\Delta^{-1} (B_{sa}^{-1}) \subseteq A_{sa}$.\smallskip

$(\!\!\!\cref{prop basic properties Jordan 6})$ It is evident from $(\!\!\!\cref{prop basic properties Jordan 4})$ and $(\!\!\!\cref{prop basic properties Jordan 5})$.\smallskip

$(\!\!\!\cref{new f-g 3.13})$ Take $a,b\in A_{sa}^{-1}$. The element $U_a(b) = Q(a) (b)$ is invertible and self-adjoint. Thus, as in the proof of $(\!\!\!\cref{prop basic properties Jordan 5})$, $Q(\Delta(U_a(b))) = Q(\Delta(Q(a)(b)))$ is injective. Since $U_a(b) = Q(a)(b)$ is a truncation of itself, we deduce from the hypotheses on $\Delta$ that 
$$\{\Delta(Q(a)(b)),\Delta(Q(a)(b)),\Delta(Q(a)(b))\} = \{\Delta(U_a(b)), \{\Delta(a), \Delta(b), \Delta(a)\},\Delta(U_a(b))\}$$ equivalently, $$Q(\Delta(Q(a)(b)))\Big(\Delta(Q(a)(b))- \{\Delta(a), \Delta(b), \Delta(a)\}  \Big) = 0,$$ and $\Delta(Q(a)(b))= \{\Delta(a), \Delta(b), \Delta(a)\},$ which gives the desired statement.  
\end{proof}

There is a geometric particularity of atomic JBW$^*$-triples which will be applied in our next arguments. For each minimal tripotent $e$ in a JBW$^*$-triple, $M,$ there exists a unique extreme point of the closed unit ball of $M_*$, $\varphi_e$, at which $e$ attains its norm and the corresponding Peirce-2 projection writes in the form $P_2 (e) (x) = \varphi_e(x) e$ for all $x\in M$ (cf. \cite[Proposition 4]{FriedRusso_Predual}).  It is proved in the just quoted reference that the mapping $$\mathcal{U}_{min} (M)\to \partial_{e} (\mathcal{B}_{M_*}), \ \  e\mapsto \varphi_e $$ is a bijection from the set of minimal tripotents in $M$ onto the set of all extreme points of the closed unit ball of $M_*$, which are also called pure atoms of $M$. When $E$ is an atomic JBW$^*$, its predual coincides with the norm closure of the the linear span of its extreme points (cf. \cite[Theorem 1]{FriedRusso_Predual}). Consequently, the extreme points of the closed unit ball of $E_*$ (i.e., the pure atoms of $E$) separate the points of $E$ in this latter case.\smallskip

We present next an identity principle which will be required in the proof of our main result. 

\begin{thrm}\label{thm of extension T and Delta}
Let $E$ and $F$ be two atomic JBW$^*$-triples. Suppose $\Delta: E \to F$ is a {\rm(}non-necessarily linear{\rm)} bijection preserving the truncation of triple products in both directions. We shall additionally assume that $\Delta(\alpha e)=\alpha \Delta(e)$ for every  minimal tripotent $e\in E$ and all $\alpha \in \CC$. Let $T:E\to F$ be a triple isomorphism satisfying $T(e)=\Delta(e)$ for every minimal tripotent $e\in E$. Then $\Delta= T$ is a linear triple isomorphism. \smallskip

If we assume that $\Delta(\alpha e)=\overline{\alpha} \Delta(e)$ for every  minimal tripotent $e\in E$ and all $\alpha \in \CC$, then $\Delta$ is a conjugate-linear triple isomorphism. 
\end{thrm}

\begin{proof} 	Since every element in $E$ is contained in the Peirce-2 subspace associated with a maximal tripotent in $E$  (see \cite[Lemma 3.12]{Horn_MathScand_1987}), it suffices to prove that $T|_{E_2(v)}=\Delta|_{E_2(v)}$ for every maximal tripotent $v\in E$.\smallskip
	
Pick a maximal tripotent $v\in E$. \cref{Peirce2toPeirce2} assures that  the restricted mapping $\Delta|_{E_2(v)}:E_2(v)\to F_2(\Delta(v))$ is a bijection preserving the truncation of triple products in both directions, which is also unital in this case.\smallskip

We claim that $T(v) = \Delta (v)$. Namely, observe that both elements are tripotents in the atomic JBW$^*$-triple $F$. Since every tripotent in $F$ is the supremum of all minimal tripotents bellow it, and $\Delta$ preserves tripotents and order among them in both directions (cf. \cref{lem pres tripotents and leq}), to show that $T(v) = \Delta(v)$ it suffices to prove that for each minimal tripotent $\tilde{w}$ in $F$ we have $\tilde{w}\leq \Delta(v)$ if and only if $\tilde {w}\leq T(v)$. For any minimal tripotent $\tilde{w}$ in $F$ there exists a minimal tripotent $e\in E$ with $\Delta(e) =\tilde {w}$ (cf. \cref{lem pres tripotents and leq}). By hypotheses $T(e) = \Delta(e)$. It is therefore easy to see, by a new application of \cref{lem pres tripotents and leq} and the fact that $T$ is a triple isomorphism, that $\tilde{w} = \Delta(e) \leq \Delta(v) \Leftrightarrow e\leq v \Leftrightarrow T(e) = \tilde{w}\leq T(v)$, which concludes the proof of the claim.\smallskip

Therefore, $\Delta|_{E_2(v)},T|_{E_2(v)}:E_2(v)\to F_2(\Delta(v))$ are unital bijections, $\Delta|_{E_2(v)}$ preserves the truncation of triple products in both directions, $T|_{E_2(v)}$ is a triple isomorphism, and $T(e) = \Delta (e)$ for all minimal tripotent in $E_2(v)$.\smallskip

We shall finally prove that $\Psi = T|_{E_2(v)}^{-1} \Delta|_{E_2(v)} = Id_{E_2(v)}$. Note that $\Psi$ enjoys the same properties of $\Delta$ and $\Psi (e) =e$, for every minimal tripotent $e\in E_2(v)$. The involution on the JB$^*$-algebra $E_2(v)$ will be simply denoted by ``$*$''. Fix an arbitrary $a\in E_2(v)$, a minimal tripotent $e\in E_2(v)$, and the supporting pure atom $\varphi_e$ in the extreme points of the closed unit ball of $E_*$ (\cite[Proposition 4]{FriedRusso_Predual}). Observe that $\varphi_e = \varphi_e P_2(v)$ also is a pure atom of ${E_2(v)}$ supporting $v$ (cf. \cite[Proposition 1]{FriedRusso_Predual}), and $P_2(e)(a) = \varphi_e(a)e.$ We split the argument into two cases:\smallskip

\textit{Case 1.} If $\varphi_e(a)\neq 0$ then $\varphi_e(a) e$ is a truncation of $a$, more concretely, $$\varphi_e(a)^2 \overline{\varphi_e(a)} e= \{\varphi_e(a) e, \varphi_e(a) e, \varphi_e(a) e\} = \{\varphi_e(a)(e), a, \varphi_e(a)(e)\}.$$ We deduce from \cref{prop basic properties Jordan}$(d)$ and the properties of $\Psi$ that $$\begin{aligned} \varphi_e(a)^2 \overline{\varphi_e(a)} e &= \varphi_e(a)^2 \overline{\varphi_e(a)} \Psi(e) =
\{\Psi(\varphi_e(a)e),\Psi(\varphi_e(a)e),\Psi(\varphi_e(a)e)\} \\ 
&=\{\Psi(\varphi_e(a)e),\Psi(a^*)^*,\Psi(\varphi_e(a)e)\} = \varphi_e(a)^2 \{\Psi(e), \Psi(a^*)^*, \Psi(e)\}\\
&= \varphi_e(a)^2 \{e, \Psi(a^*)^*, e\}= \varphi_e(a)^2 \overline{\varphi_e(\Psi(a^*)^*)} e,
\end{aligned} $$  which implies that 
$$\varphi_e (a) =\varphi_e(\Psi(a^*)^*)\Leftrightarrow \varphi_e(a-\Psi(a^*)^*)=0.$$

\textit{Case 2.} If $\varphi_e (a)= 0.$ This means that $P_2(e)(a)=0,$ equivalently, $\{e,a,e\} = 0\Leftrightarrow \{e^*,a^*,e^*\} =0$, and thus $a^*\in (E_2(v))_1(e^*)\oplus (E_2(v))_0(e^*).$ By \cref{lem preserve annihil} or \cref{c preservation Perice 1 + 0}, we get $$\Psi(a^*)\in (E_2(v))_1(\Psi(e^*))\oplus (E_2(v))_0(\Psi(e^*)) = (E_2(v))_1(e^*)\oplus (E_2(v))_0(e^*),$$ or equivalently, $$\{e^*,\Psi(a^*),e^*\}=0 \Leftrightarrow \{e,\Psi(a^*)^*,e\}=0 \Leftrightarrow  \varphi_e(\Psi(a^*)^*-a)=0.$$

In both cases we arrived to  $\varphi_e(a-\Psi(a^*)^*)=0,$ and this holds 
for every minimal tripotent $e\in E_2(v).$ Therefore $\Psi(a^*)^*=a$, since the pure atoms of $E_2(v)$ separate the points of this space.  The arbitrariness of $a\in E_2(v)$ implies that $\Psi_{|E_2(v)}=Id_{E_2(v)},$ and by the comments at the beginning of the proof we have $\Delta=T.$
\end{proof}

It follows from the above identity principle that it is necessary to determine the form of the maps $f_e$ associated with each minimal tripotent $e$ defined in page~\pageref{def of fe}.

\subsection{Preservers of the truncation of triple products on spin factors}\label{subsec:spin factors} \ \smallskip

In this subsection we focus on a very special class of atomic JBW$^*$-algebras, the 3-dimensional spin factor $S_{2}(\CC)$ of all symmetric 2 by 2 complex matrices. This spin factor plays a central role in classic and recent results (see, for instance, \cite[Sect.7]{Alfsen_Shultz_Stomer_AdvMath_1978}, \cite[page 82]{FriedRusso_Predual} or \cite[Sect.3]{Kal_Peralta_Ann_Math_Phys_2021}). \smallskip

Let us give some hints about our interest in this special Cartan factor. It is known that if $e_1$ and $e_2$ are two orthogonal minimal tripotents in a JB$^*$-triple $E$, the Peirce-2 subspace $E_2(e_1+e_2)$ is isometrically triple isomorphic to $\mathbb{C}\oplus^{\ell_{\infty}} \mathbb{C}$ or to a spin factor (cf. \cite[Lemma 3.6]{Kal_Peralta_Ann_Math_Phys_2021}). If $E$ is a Cartan factor, the case of $\mathbb{C}\oplus^{\ell_{\infty}} \mathbb{C}$ is simply impossible (compare, for example, \cite[Lemma 3.10]{Polo_Peralta_AdvMath_2018}). \smallskip

Let $\Delta: E\to F$ be a bijection preserving the truncation of triple products in both directions between two atomic JBW$^*$-triples. Suppose we can find two orthogonal minimal tripotents $e_{1}$ and $e_{2}$ both contained in the same Cartan factor inside $E$. As commented above  $E_2(e_1+e_2)$ is a Spin factor. Note that, by \cref{Peirce2toPeirce2} and \cref{lem orthogonal tripotents}, $\Delta$ maps $E_{2}(e_1+e_2)$ onto $F_{2} (\Delta(e_{1}) + \Delta(e_{2})),$ which must be also a spin factor as we shall see next.\smallskip

Let us recall some results on the structure of the lattice of tripotents in a spin factor borrowed from \cite{Fried_Peralta_Ann_Math_Phys_2022, Kal_Peralta_Ann_Math_Phys_2021}. A \emph{spin factor}, is complex Hilbert spaces $X$ provided with a conjugation (i.e. conjugate-linear isometry of period-2) $x\mapsto \overline{x},$ whose triple product
and norm are defined by \begin{equation}\label{eq spin product}
	\{x, y, z\} = \langle x|y\rangle z + \langle z|y\rangle  x -\langle x|\overline{z}\rangle \overline{y},
\end{equation} and \begin{equation}\label{eq spin norm} \|x\|^2 = \langle x|x\rangle  + \sqrt{\langle x|x\rangle ^2 -|
		\langle x|\overline{x}\rangle  |^2},
\end{equation} respectively. We shall always assume that dim$(X)\geq 3$, since in dimensions $1$ and $2$ we have $X = \mathbb{C}$ and $X= \mathbb{C}\oplus^{\ell_{\infty}} \mathbb{C}$, respectively, and the latter is not a factor.\smallskip

The real Hilbert subspace 
$X^{-} =\{x\in X : \overline{x} =x\}$ serves to understand all complete or maximal tripotents of $X$. Observe first that  $\|x\|= \|x\|_2 = \sqrt{\langle x|x\rangle}$ for all $x\in X^{-}$ and  $\langle a | b\rangle = \langle b| a\rangle \in \mathbb{R}$ for all $a,b\in X^{-}$.
We further have $X= X^{-} \oplus i X^{-}$. Every maximal tripotent on $X$ is of the form $u=\gamma x$ with $x$ in the unit sphere of $X$ and $\gamma$ a unitary in $\mathbb{C}$. Given two norm-one elements $x,y\in X^{-}$ with $\langle x| y\rangle =0,$ the elements $e_1= \frac{x+iy}{2}$ and $e_2= \frac{x-iy}{2}$ are two mutually orthogonal minimal tripotents in $X$ with $x = e_1+e_2$. It can be also checked that  $$\begin{aligned}X_1(e_1) =X_1(e_2)=& \left\{{c+id}: c,d\in X^{-}, \ c,d \perp_2 x,y \right\}:=\{x,y\}^{\perp_2}_X, \\ X_2 (e_1)&= \mathbb{C} e_1, \hbox{ and } X_0 (e_1)= \mathbb{C} e_2,
\end{aligned}$$ where  $\perp_2$ is used to denote  orthogonality in the Hilbert space $(X,\langle\cdot|\cdot\rangle)$. In particular $X$ has rank-2. Every minimal tripotent in $X$ can be obtained in the way just commented. Since dim$(X)\geq 3$, we can find three mutually orthogonal (in the Hilbert sense) norm-one elements $a,b,c$ in $X^{-}$. It is known, and easy to check, that the JB$^*$-subtriple of $X$ generated by $\{a,b,c\}$, equivalently, by $\{c, e_1=\frac{a+ib}{2}, e_2=\frac{a-ib}{2}\}$ is a 3-dimensional spin factor JB$^*$-triple isomorphic to $S_2$ by just identifying $c \cong \left( \begin{array}{cc}
0 & 1 \\
1 & 0
\end{array}\right)$, $e_1 \cong \left( \begin{array}{cc}
1 & 0 \\
0 & 0
\end{array}\right),$ and $e_2 \cong\left( \begin{array}{cc}
0 & 0 \\
0 & 1
\end{array}\right)$. In case that we can find an orthonormal system of the form $\{a,b,c,d\}$ in $X^{-}$, it is also known that  the JB$^*$-subtriple of $X$ generated by $\{a,b,c,d\}$, equivalently, by $\{ e_1=\frac{a+ib}{2}, e_2=\frac{a-ib}{2},e_3=\frac{c+id}{2}, e_4=\frac{c-id}{2}\}$ is triple isomorphic to $M_2$ by just identifying $e_1 \cong \left( \begin{array}{cc}
1 & 0 \\
0 & 0
\end{array}\right),$  $e_2 \cong \left( \begin{array}{cc}
0 & 0 \\
0 & 1
\end{array}\right)$,  $e_3 \cong \left( \begin{array}{cc}
0 & 1 \\
0 & 0
\end{array}\right)$, and $e_4 \cong\left( \begin{array}{cc}
0 & 0 \\
1 & 0
\end{array}\right)$. \smallskip

Let us observe that for each minimal tripotent $e$ in a spin factor $X$ with dim$(X)\geq 4$, its orthogonal complement reduces to $\mathbb{C} \overline{e}= X_0(e)$ and we further have \begin{equation}\label{equation fla for Peirce 02} \mathbb{C} \overline{e}\oplus \mathbb{C} {e} =  X_0(e) \oplus X_2 (e) =  \bigcap\left\{ X_{1} (v) : v\in X_1(e) \hbox{ minimal in } X \right\}.
\end{equation} We note that the second equality in \eqref{equation fla for Peirce 02} does not hold for $S_2$. \smallskip

In the next proposition we elaborate the necessary observations to reduce our study to spin factors of the form $S_2$ or $M_2$. 

\begin{prop}\label{r preservers of truncations between spin factors}
Let $\Delta: X\to Y$ be a {\rm(}non-necessarily linear{\rm)} bijection preserving the truncation of triple products in both directions between two spin factors. We shall employ the same symbols $\langle \cdot| \cdot\rangle$ and $\overline{\ \cdot \ }$ to denote the inner products and involutions of $X$ and $Y$. The following properties hold:
\begin{enumerate}[$(a)$]
\item For each $\gamma\in \mathbb{T}$ and each norm-one $x\in X^{-}$ there exist $\gamma{'}\in \mathbb{T}$ and a norm-one $y\in Y^{-}$ such that $\Delta (\gamma x) =\gamma{'} y$. 
\item  For each orthonormal set $\{a,b\}$ in $X^{-}$ there exists a unique orthonormal set $\{a',b'\}$ in $Y^{-}$ satisfying $\Delta\left(\frac{a+ib}{2}\right)= \frac{a'+ib'}{2}$. Moreover, for $e_1=\frac{a+ib}{2}$, $X_1 (e_1)\oplus X_0(e_1) = \{a,b\}_{X}^{\perp_2} \oplus \mathbb{C} \frac{a-ib}{2}$, and hence $$\Delta\left( \{a,b\}^{\perp_2}_{X} \oplus \mathbb{C} \frac{a-ib}{2} \right)  = \{a',b'\}_{Y}^{\perp_2} \oplus \mathbb{C} \frac{a'-ib'}{2}.$$
\item For orthonormal set $\{a,b\}$ in $X^{-}$ with $e_1 = \frac{a+ib}{2}$,  $\Delta(e_1)=\Delta\left(\frac{a+ib}{2}\right)= \frac{a'+ib'}{2},$ where $\{a',b'\}$ is  an orthonormal set in $Y^{-}$, we have $$\begin{aligned}
\Delta\left(X_1(e_1) \right)=\Delta (\{a,b\}_{X}^{\perp_2}) &= \{a',b'\}_{Y}^{\perp_2}=Y_1(\Delta(e_1))\\ &=\left\{\Delta\left(\frac{a+ib}{2}\right), \Delta\left(\frac{a-ib}{2}\right)\right\}^{\perp_2}_{Y} \\
\hbox{ and } \Delta\left(X_0(e_1)\right) = Y_0 (\Delta(&e_1)).
\end{aligned}$$     
\item For each orthonormal system $\{a,b,c\}$ in $X^{-}$ one of the next statements holds: 
\begin{enumerate}[$(d.1)$]
	\item $\Delta$ maps the JB$^*$-subtriple $S\cong S_2$ of $X$ generated by $\{a,b,c\}$ onto a JB$^*$-subtriple of $Y$ isometrically isomorphic to $S_2$;
	\item We can find an orthonormal system of the form $\{a,b,c,d\}$ in $X^{-}$ such that $\Delta$ maps the JB$^*$-subtriple $M\cong M_2$ of $X$ generated by $\{a,b,c,d\}$ onto a JB$^*$-subtriple of $Y$ isometrically isomorphic to $M_2$.
\end{enumerate} Consequently, given any two minimal tripotents $e,v$ in $X$, there exist JB$^*$-subtriples $M\subseteq X$ and $N\subseteq Y$ satisfying $e,v\in M $, $\Delta(M)= N$, and $M,N\cong M_2$ or $M,N\cong S_2$.   
\end{enumerate}	
\end{prop}

\begin{proof} $(a)$ follows from \cref{lem pres tripotents and leq}, while the same lemma and \cref{c preservation Perice 1 + 0} give $(b)$.\smallskip
	
$(c)$ Take any $z\in \{a,b\}_{X}^{\perp_2}$. By $(b),$ $\Delta(z) = z_1' + \lambda \frac{a'-ib'}{2} = z_2' + \mu \frac{a'+ib'}{2}$, for some $z_1',z_2'\in  \{a',b'\}_{Y}^{\perp_2}$, $\lambda, \mu \in\mathbb{C},$ and thus $\Delta (z)\in \{a',b'\}_{Y}^{\perp_2}.$ This shows that $\Delta\left(X_1(e_1) \right)=Y_1(\Delta(e_1))$. If $z\in X_0(e_1) = \mathbb{C} e_2$ with $e_2=\frac{a-ib}{2}$, we can write $z = \lambda e_2$ for some complex number $\lambda$. By applying the function $f_{e_2}$ (see page~\pageref{def of fe}) we get $\Delta (z) = f_{e_2} (\lambda) \Delta(e_2)$, and by \cref{lem orthogonal tripotents} that $\Delta(e_2)$ is a tripotent orthogonal to $\Delta(e_1)$, and therefore $\Delta (e_2) = \gamma \frac{a'-ib'}{2}$ for some unitary $\gamma\in \mathbb{C}$. Therefore $\Delta(z) \in Y_0(\Delta(e_1))$.\smallskip

$(d)$ To prove the desired statement, set $e_1= \frac{a+ib}{2}$ and write  $\Delta\left(\frac{a+ib}{2}\right)= \frac{a'+ib'}{2}$, 
where $\{a',b'\}$  is an orthonormal system in $Y^{-}$. We shall distinguish three cases:\smallskip

\emph{Case 1}: dim$(X) =3$. If this holds, $X\cong S_2$, and $X_1(e_1) $ $= \mathbb{C} c$ is $1$-dimensional. The properties of the mapping $f_{c}$ (see the definition in page~\pageref{def of fe}) now imply that $\Delta\left(X_1(e_1)\right)= \Delta (\mathbb{C}c) = f_{c} (\mathbb{C}) \Delta(c).$ Thus, it follows from $(c)$ that $Y_1(\Delta(e_1))$ also is $1$-dimensional. This holds if and only if $Y$ is three dimensional and necessarily $Y\cong S_2$. Therefore $\Delta$ maps $S_2\cong X$ onto $Y\cong S_2$. \smallskip

\emph{Case 2}: dim$(X) =4$. It is also known, and not hard to see, that $X$ is triple isomorphic to $M_2,$ the space of all $2\times 2$ matrices with complex entries. We also know that the number of mutually orthogonal elements in the Hilbert space $X$ cannot exceed 4. We shall first show that $Y$ must be also $4$-dimensional. If dim$(Y)=3,$ we apply the previous \emph{Case 1} to $\Delta^{-1}: Y\to X$ and we obtain that dim$(X)=3$, which is impossible. If dim$(Y)\geq 5$, we can find an orthonormal system $\{\xi_1,\xi_2,\xi_3,\xi_4,\xi_5\}$ in $Y^{-}$. Having in mind $(c)$ we get that $$\begin{aligned}
	\Delta^{-1}(\xi_j), \Delta^{-1}\left(\frac{\xi_3+i\xi_4}{2}\right), \  &\Delta^{-1}\left(\frac{\xi_3-i\xi_4}{2}\right)\in X_1\left(\Delta^{-1}\left(\frac{\xi_1+i\xi_2}{2}\right)\right)  \\
	&= \left\{\Delta^{-1}\left(\frac{\xi_1+i\xi_2}{2}\right), \Delta^{-1}\left(\frac{\xi_1-i\xi_2}{2}\right) \right\}^{\perp_2}_X,
\end{aligned}$$ for all $j =3,4,5$ and thus $
\{\Delta^{-1}\left(\xi_5\right)\}\cup X_2\left(\Delta^{-1}\left(\frac{\xi_3+i\xi_4}{2}\right)\right)\oplus X_0\left(\Delta^{-1}\left(\frac{\xi_3+i\xi_4}{2}\right)\right) $ is contained in $X_1\left(\Delta^{-1}\left(\frac{\xi_1-i\xi_2}{2}\right)\right) = X_1\left(\Delta^{-1}\left(\frac{\xi_1+i\xi_2}{2}\right)\right).$ It also follows from $(c)$ that $$ \Delta^{-1}(\xi_5)\in \left\{\Delta^{-1}\left(\frac{\xi_3+i\xi_4}{2}\right), \Delta^{-1}\left(\frac{\xi_3-i\xi_4}{2}\right) \right\}^{\perp_2}_X = X_1\left(\Delta^{-1}\left(\frac{\xi_3+i\xi_4}{2}\right)\right).$$ Therefore, the elements in the following three sets must be orthogonal in the Hilbert space $X$: $\{\Delta^{-1}\left(\xi_5\right)\}$, $X_2\left(\Delta^{-1}\left(\frac{\xi_3+i \xi_4}{2}\right)\right)\oplus X_0\left(\Delta^{-1}\left(\frac{\xi_3+i \xi_4}{2}\right)\right)$, and \linebreak $X_2\left(\Delta^{-1}\left(\frac{\xi_1+i \xi_2}{2}\right)\right)\oplus X_0\left(\Delta^{-1}\left(\frac{\xi_1+i \xi_2}{2}\right)\right),$ where the last two of them contain at least two elements which are orthogonal in the Hilbert space $X$. This implies that dim$(X)\geq 5$, which contradicts our assumption.\smallskip

We have therefore shown that dim$(Y) =4$, and thus $X,Y\cong M_2.$ Observe that in this case $\Delta: M_2\to M_2$ is a bijective mapping preserving the truncation of triple products in both ways.  \smallskip

\emph{Case 3}: dim$(X) \geq 5$. Since $\Delta\left(\frac{a-ib}{2}\right) \in Y_0 \left(
\Delta\left(\frac{a+ib}{2}\right) \right) = \mathbb{C} \frac{a'-ib'}{2}$, and thus  $\Delta\left(\frac{a-ib}{2}\right) = \gamma_{a,b} \frac{a'-ib'}{2}$ for some unitary $\gamma_{a,b}\in \mathbb{C}$, it is easy to check that 
$$\begin{aligned}
	Y_1\left(\Delta\left(\frac{a+ib}{2}\right) \right) &= \left\{ \Delta\left(\frac{a+ib}{2}\right), \Delta\left(\frac{a-ib}{2}\right) \right\}^{\perp_2}_Y\\
	&= \left\{ \frac{a'+ib'}{2}, \gamma_{a,b} \frac{a'-ib'}{2} \right\}^{\perp_2}_Y = \left\{ a',b' \right\}^{\perp_2}_Y.
\end{aligned}$$
Let us find an orthonormal system $\{a,b,c,d\}$ in $X^{-}$, and we write $\Delta\left(\frac{c+id}{2}\right)= \frac{c'+id'}{2}$, where $\{c',d'\}$ is an orthonormal system in $Y^{-}$. As before, $Y_1\left( \Delta\left(\frac{c+id}{2}\right) \right)=\left\{ \Delta\left(\frac{c+id}{2}\right), \Delta\left(\frac{c-id}{2}\right) \right\}^{\perp_2}_Y = \left\{c',d'\right\}^{\perp_2}_Y.$ Furthermore, having in mind that $\Delta\left(\frac{a+ib}{2}\right),$ $ \Delta\left(\frac{a-ib}{2}\right)\in \left\{\Delta\left(\frac{c+id}{2}\right),\Delta\left(\frac{c-id}{2}\right) \right\}^{\perp_2}_Y$, it follows that $\{a',b',c',d'\}$ is an orthonormal system in $Y^{-}$.\smallskip

We show next that $\Delta \left(\{a,b,c,d\}_{X}^{\perp_2} \right)= \{a',b',c',d'\}^{\perp_2}_{Y}.$ The desired conclusion follows from $(c)$ by just observing that $$\begin{aligned}
	\Delta \left(\{a,b,c,d\}_{X}^{\perp_2} \right) &=  \Delta \left( X_1\left(\frac{a+ib}{2} \right) \cap X_1\left( \frac{c+id}{2}\right)  \right) \\
	&= Y_1\left( \Delta\left(\frac{a+ib}{2}\right) \right)\cap  Y_1\left( \Delta\left(\frac{c+id}{2}\right) \right) \\
	&=\{a',b'\}_{Y}^{\perp_2}\cap \{c',d'\}_{Y}^{\perp_2} = \{a',b',c',d'\}_{Y}^{\perp_2}.
\end{aligned} $$

If dim$(X)$ is even and finite, we find an orthonormal basis of $X^{-}$ of the form $\{a_1,\ldots,a_n,b_1,\ldots,b_n \}$ ($n\geq 3$). By considering the minimal tripotents $e_{j} =\frac{a_j + i b_j}{2}$ ($j\in \{1,\ldots, n\}$), and its images in the form $\Delta\left(\frac{a_j + i b_j}{2}\right) =\frac{a_j' + i b_j'}{2}$, and the arguments above show that the set $\{a_1',\ldots,a_n',b_1',\ldots,b_n' \}$ is an orthonormal basis of $Y^{-}$. We further have 
$$\begin{aligned} \Delta (M_2)&\cong	\Delta \left( \hbox{span} \left\{\frac{a_j+ib_j}{2}, \frac{a_j-ib_j}{2} : j\leq 2\right\} \right) = 
	\Delta \left(\{a_j,b_j : j\geq 3\}_{X}^{\perp_2} \right)\\ 
	&=  \Delta \left(\bigcap_{j\geq 3} X_1\left(\frac{a_j+ib_j}{2} \right)\right) = \bigcap_{j\geq 3}\left( Y_1\left( \Delta\left(\frac{a_j+ib_j}{2}\right) \right) \right) \\
	&=\{a_j',b_j'; j\geq 3\}_{Y}^{\perp_2}= \hbox{span} \left\{\frac{a'_j+ib'_j}{2}, \frac{a'_j-ib'_j}{2} : j\leq 2\right\}\cong M_2.  
\end{aligned} $$ In case that $X$ is infinite dimensional, we can proceed similarly with the help of the Cantor-Berstein theorem.\smallskip

Finally, if dim$(X) = 2n +1$ is odd, we pick an orthonormal basis of $X^{-}$ of the form $\{a_1,\ldots,a_n,b_1,\ldots,b_n \}\cup\{c\}$ ($n\geq 4$). As before the minimal tripotents $e_{i} =\frac{a_i + i b_i}{2}$ ($i\in \{1,\ldots, n\}$), satisfy that $\Delta\left(\frac{a_i + i b_i}{2}\right) =\frac{a_i' + i b_i'}{2}$, where $\{a_1',\ldots,a_n',b_1',\ldots,b_n' \}$ is an orthonormal system of $Y^{-}$. The element $\Delta(c) = \gamma c'$ for some norm-one $c'\in Y^{-}$ which is orthogonal to all $a_j',b_j'$ by $(c)$. If $\{a_1',\ldots,a_n',b_1',\ldots,b_n' \}\cup\{c'\}$ is not an orthonormal basis of $Y^{-}$, we could extend it to an orthonormal basis with at least one more element, which by the previous discussion would imply that $X$ has dimension at least $2 n +2$, which is impossible. We conclude by observing that 
$$\begin{aligned} \Delta (S_2)&\cong \Delta \left( \hbox{span} \left\{\frac{a_j+ib_j}{2}, \frac{a_j-ib_j}{2},c : j\leq 2\right\} \right) = 
	\Delta \left(\{a_j,b_j : j\geq 3\}_{X}^{\perp_2} \right)\\ 
	&=  \Delta \left(\bigcap_{j\geq 3} X_1\left(\frac{a_j+ib_j}{2} \right)\right) = \bigcap_{j\geq 3}\left( Y_1\left( \Delta\left(\frac{a_j+ib_j}{2}\right) \right) \right) \\
	&=\{a_j',b_j'; j\geq 3\}_{Y}^{\perp_2}= \hbox{span} \left\{\frac{a'_j+ib'_j}{2}, \frac{a'_j-ib'_j}{2}, c : j\leq 2\right\}\cong S_2.  
\end{aligned} $$ 
\end{proof}

\begin{rem}\label{r our result doesnt follow from the case of B(H)} There is a natural connection between the problem studied in this note and the one in \cite{Jia_Shi_Ji_AnnFunctAnn_2022} for $B(H)$. In the just quoted reference the authors study bijections $F : B(H)\to B(H)$ with the following property: $a$ is a truncation of $bcb$ if and only if $\Delta (a)$ is a truncation of $\Delta (b) \Delta(c) \Delta (b)$. The meaning of truncation does not change when $B(H)$ is regarded as a JB$^*$-triple, while our mapping $\Delta : B(H)\to B(H)$ is a bijection satisfying:  $a$ is a truncation of $bc^*b$ if and only if $\Delta (a)$ is a truncation of $\Delta (b) \Delta(c)^* \Delta (b)$, which is a different hypothesis non-equivalent a priori to the previous one. So, in principle, in the setting of $B(H)$ spaces the hypotheses here and in \cite{Jia_Shi_Ji_AnnFunctAnn_2022} are not directly related, however the conclusions are the same (compare \cite[Theorem 3.1]{Jia_Shi_Ji_AnnFunctAnn_2022} and \cref{t main thrm preservers of truncations of triple products}).
\end{rem}

The next lemma assures that we can restrict our study to certain concrete minimal tripotents.  

\begin{lem}\label{l intechange of tripotents}
Let $E$ and $F$ be JB$^*$-triples, and let $e$ be a minimal tripotent in $E$ satisfying the following property: for every bijection $\tilde{\Delta}: E\to F$ preserving the truncation of triple products in both directions we have $\tilde{\Delta} (\lambda e) = \lambda \tilde{\Delta}(e)$ for all $\lambda\in \mathbb{R}$ {\rm(}respectively, for all $\lambda\in \mathbb{C}${\rm)}. Suppose $v$ is a minimal tripotent in $E$ and $\Delta: E\to F$ is a bijection preserving the truncation of triple products in both directions such that there exits a triple isomorphism $\Phi: E\to E$ with $\Phi(e) =v$. Then ${\Delta} (\lambda v) = \lambda \Delta(v),$ for all $\lambda\in \mathbb{R}$ {\rm(}respectively, for all $\lambda\in \mathbb{C}${\rm)}. 
\end{lem}

\begin{proof} Observe that the composition $\Delta \Phi : E\to F$ also is a bijection preserving the truncation of triple products in both directions with $\Delta (\lambda v) = \Delta \Phi (\lambda e) = \lambda \Delta \Phi (e) = \lambda \Delta\Phi (e) = \lambda \Delta(v),$ for all $\lambda \in \mathbb{R}$ (respectively, in $\mathbb{C}$), by hypotheses.
\end{proof}

In the rest of this subsection we shall study the cases in which $\Delta : S_2\to S_2$ and $\Delta : M_2\to M_2$. Both JB$^*$-triples, $M_2$ and  $S_2$ are spin factors having rank-two. Let us find two orthogonal minimal tripotents $e_1,e_2$ in $S_2$ or in $M_2,$ and  write $u = e_1 + e_2$ for the corresponding unitary tripotent in the spin $S_2$ or $M_2$. It follows from \cref{lem almost orth add}, \cref{lem orthogonal tripotents}, \cref{Peirce2toPeirce2} (see also page~\pageref{def of fe})  and \cref{cor of fmultiplicative} that there are functions $f_{e_{1}}$ and $f_{e_{2}}$ on $\CC$ (depending on $e_{1}$ and $e_{2}$, respectively) such that
$$\begin{aligned}
\Delta (\lambda e_{1} +  \mu e_{2})^{[3]} &= \Delta ((\lambda^{[3]} e_{1})^{[\frac13]} +  (\mu^{[3]} e_{2})^{[\frac13]})^{[3]} = \Delta (\lambda^{[3]} e_{1}) + \Delta (  \mu^{[3]} e_{2}) \\ 
&= f_{e_{1}}(\lambda)^{[3]} \Delta(e_{1}) + f_{e_{2}}(\mu)^{[3]} \Delta(e_{2}) \quad (\forall \lambda,\mu \in \CC).
\end{aligned}$$ Therefore 
\begin{equation}\label{eq additivity of Delta on Ce1 + Ce2} \Delta (\lambda e_{1} +  \mu e_{2}) = f_{e_{1}}(\lambda) \Delta(e_{1}) + f_{e_{2}}(\mu) \Delta(e_{2}) \quad (\forall \lambda,\mu \in \CC),
\end{equation} and \begin{equation}\label{eq additivity of Delta on Ce1 + Ce2 two}\Delta (\lambda u) = \Delta (\lambda e_{1} +  \lambda e_{2}) = f_{e_{1}}(\lambda) \Delta(e_{1}) + f_{e_{2}}(\lambda) \Delta(e_{2}) \quad (\forall \lambda \in \CC).\end{equation}
Our goal will consist in proving that  $f_{e_{1}} = f_{e_{2}}$ and they both are nothing but the identity map or the conjugation on $\mathbb{C}$.\smallskip

Let $e$ and $v$ be two minimal tripotents in a JBW$^*$-triple $M$. According to \cite{Per_RIMS_2023,Peralta_ResMath_2023}, the \emph{triple transition pseudo probability} from $e$ to $v$ is defined by $TTP (e,v) = \varphi_v (e)$, where $\varphi_v$ is the unique pure atom in $M_*$ supported at $v$. When $e$ and $v$ are minimal projections in $B(H)$, $TTP(e,v)$ is the usual transition probability between $e$ and $v$ in Wigner's theorem. \smallskip

The next result is probably the central core of our argument. The proof, which is quite technical, is devoted to get the conditions to apply \cite[Theorem 4.3]{Fried_Peralta_Ann_Math_Phys_2022} and the identity principle stated in \cref{thm of extension T and Delta}.

\begin{thrm}\label{linear at identity real} Let $\Delta : C\to C$ be a {\rm(}non-necessarily linear{\rm)} bijection preserving the truncation of triple products in both directions, where $C= S_2$ or $M_2$. Then one of the next statements holds: 
\begin{enumerate}[$(a)$]
\item For each minimal tripotent $e$ in $C$ we have $\Delta (\lambda e) = \lambda \Delta(e)$ for every $\lambda \in \mathbb{C}$.
\item For each minimal tripotent $e$ in $C$ we have $\Delta (\lambda e) = \overline{\lambda} \Delta(e)$ for every $\lambda \in \mathbb{C}$.
\end{enumerate}	
Consequently, $\Delta$ is a linear or a conjugate-linear triple automorphism on $C$.
\end{thrm}

\begin{proof} It is known that for $C=S_2$ or  $M_2$ we can always find a triple automorphism on $C$ interchanging any two couples of minimal tripotents (see, for example, \cite[Proposition 5.8]{Kaup_ManMath_1997}, though the statement here is much easier). So, by \cref{l intechange of tripotents}, it suffices to prove that for $e_{1} := \left( \begin{array}{cc}
		1 & 0 \\
		0 & 0
	\end{array}  \right),$ $\Delta(\lambda e_1) = \lambda \Delta(e_1),$ for every bijection $\Delta : C\to C$ preserving the truncation of triple products in both directions and every $\lambda \in \mathbb{R}$. Fix any such a mapping $\Delta$. \smallskip
	
Set $e_{2} := \left( \begin{array}{cc}
			0 & 0 \\
			0 & 1
		\end{array} \right)$ and $I = e_1 +e_2$. Let $\Phi : C\to C$ be a triple automorphism mapping $\Delta(e_j)$ to $e_j$ for  $j = 1,2$. Observe that $\Delta_1 = \Phi\Delta$ is a bijection preserving the truncation of triple products in both directions with $\Delta_1(e_j) = e_j,$ for every $j\in\{1,2\}$. Henceforth we deal with the mapping $\Delta_1$.\smallskip
		
We deal first with $C = S_2$. It is not hard to check that every minimal projection $p$ in $S_{2}$ can be written in the form
\begin{equation}\label{minimal projection}
	p = \left( \begin{array}{cc}
		\alpha & \sqrt{\alpha (1-\alpha)} \\
		\sqrt{\alpha(1-\alpha)} & 1-\alpha
	\end{array}  \right),  \quad \text{for some $\alpha\in [0,1]$}. 
\end{equation}

Since $ \alpha e_{1}$ is a truncation of $p$ for every $\alpha \in [0,1],$ it follows that $\Delta_1 ( \alpha e_{1}) = f_{e_{1}}(\alpha) e_{1}$ is a truncation of $\Delta_1 (p)$ (cf. \cref{r Delta preserves the truncation of tripotents}), where the latter is another projection in $S_2$ (cf. \cref{prop basic properties Jordan}), which implies that $f_{e_{1}}|_{[0,1]} : [0,1] \to [0,1]$. It is easy to check from the injectivity of $\Delta_1,$ and the properties of $f_{e_1},$ that $f_{e_{1}}|_{[0,1]}$ is injective. Since $\Delta_1^{-1}$ satisfies the same properties, there exists a bijection $g_{e_1} : \mathbb{C}\to \mathbb{C}$ with the same properties of $f_{e_1}$ satisfying $\Delta_1^{-1} (\lambda e_1) = g_{e_1}(\lambda) e_1$, and thus $f_{e_1} g_{e_1} (\lambda) e_1 = \Delta_1 \Delta_1^{-1} (\lambda e_1) = \lambda e_1 =  \Delta_1^{-1}  \Delta_1 (\lambda e_1) =  g_{e_1} f_{e_1} (\lambda) e_1,$ for all $\lambda$ in $\mathbb{C}$. Since $g_{e_1} ([0,1]) \subseteq [0,1]$, $f_{e_{1}}|_{[0,1]} : [0,1] \to [0,1]$ is a bijection. These same properties are also valid for $f_{e_{2}}$.\smallskip

We shall next show that $\Delta_1 (\lambda  I) =f_{e_{1}} (\lambda) \Delta_1(I) = f_{e_{2}} (\lambda) \Delta_1(I),$ for all $\lambda \in [0,1]$. We can clearly assume that $\lambda \in (0,1)$. It is clear that \begin{equation*}
	p_{1} = \left( \begin{array}{cc}
		\lambda  & \sqrt{\lambda(1-\lambda)} \\
		\sqrt{\lambda(1-\lambda)} & 1-\lambda
	\end{array}  \right), \ \  
	p_{2} = \left( \begin{array}{cc}
		1-\lambda  & \sqrt{\lambda(1-\lambda)} \\
		\sqrt{\lambda(1-\lambda)} & \lambda
	\end{array}  \right)
\end{equation*} 
are both minimal projections in $S_2$. We may assume (cf. \cref{minimal projection}) that $\Delta_1 (p_{1})$ and $\Delta_1 (p_{2})$ have the following matrix expression:
\begin{equation*}
	\Delta_1 (p_{1}) = \left( \begin{array}{cc}
		t  & \sqrt{t(1-t)} \\
		\sqrt{t(1-t)} & 1-t
	\end{array}  \right),\ \  
	\Delta_1 (p_{2}) = \left( \begin{array}{cc}
		1-t  & \sqrt{t(1-t)} \\
		\sqrt{t(1-t)} & t
	\end{array}  \right) 
\end{equation*} for some $t \in (0,1),$ because $\Delta_1$ is unital and preserves orthogonality (cf. \cref{prop basic properties Jordan}). Therefore, using the facts that $\lambda e_{1}$ and $(1-\lambda ) e_{2}$ are both truncations of $p_{1},$ and $(1 -\lambda )e_{1}$ is a truncation of $p_{2}$ and $\Delta$ preserves the truncation of tripotents (see \cref{r Delta preserves the truncation of tripotents}), we have $f_{e_{1}}(\lambda ) = t$, $f_{e_{2}}(1 - \lambda) = 1 -t,$ and $f_{e_{1}}(1 - \lambda) = 1-t$. Hence, $f_{e_{2}} (\lambda) = f_{e_{1}} (\lambda) = \lambda$ for all $\lambda \in[0,1]$, \cref{l suffcient cond for fe identity} assures that $f_{e_{2}} (\lambda) = f_{e_{1}} (\lambda) = \lambda$ for all $\lambda \in \mathbb{R}.$ The arbitrariness of $\Delta_1$ (equivalently of $\Delta$) proves that $$  \Delta (\lambda e) =\lambda \Delta (e), \ \ (\forall \lambda\in \mathbb{R}),$$ for all minimal tripotent  $e\in C=S_2$, every bijection $\Delta: S_2\to S_2$ preserving the truncation of triple products, and all $\lambda \in \mathbb{R}.$\smallskip

We have some additional conclusions on the mapping $\Delta_1$. Observe that, according to \eqref{minimal projection}, every minimal projection $p =\left( \begin{array}{cc}
	\alpha & \sqrt{\alpha (1-\alpha)} \\
	 \sqrt{\alpha(1-\alpha)} & 1-\alpha
\end{array}  \right)$ in $S_2$ is uniquely determined by a parameter $\alpha\in [0,1]$. We know that $\Delta_1 (p)$ is a minimal projection in $S_2$ (cf. \cref{prop basic properties Jordan}) with $f_{e_1} (\Delta_1 (p)) = \alpha,$ therefore 
\begin{equation}\label{eq Delta1 fixes min projections S_2} \hbox{$\Delta_1 (p) =p,$ for every minimal projection $p$ in $S_2$.}
\end{equation}

Let us make the necessary changes to deal with $C = M_2$. First, the minmal projections in $M_2$ can be represented in the form \begin{equation}\label{minimal projection in M2}
	p = \left( \begin{array}{cc}
		\alpha & \gamma\sqrt{\alpha (1-\alpha)} \\
		\overline{\gamma} \sqrt{\alpha(1-\alpha)} & 1-\alpha
	\end{array}  \right),  \quad \text{for some $\alpha\in [0,1],$ $\gamma\in \mathbb{T}$.} 
\end{equation} This small change in the representation does not imply any difference in the arguments, which remain valid in this case and thus  $$  \Delta (\lambda e) =\lambda \Delta (e),$$ for all minimal tripotent  $e\in C=M_2$, every bijection $\Delta: M_2\to M_2$ preserving the truncation of triple products, and all $\lambda \in \mathbb{R}.$\smallskip

Summarizing what we proved until now, suppose $u, u_1,u_2$ are tripotents in $C$ with $u$ unitary, $u_1,u_2$ minimal and $u= u_1 + u_2$. We deduce from \eqref{eq additivity of Delta on Ce1 + Ce2} and the previous conclusions that \begin{equation}\label{eq Delta behaves real linearly on real combinations of unitaries} \Delta_1 (\lambda u_1 + \mu u_2) = \Delta_1 (\lambda u_{1}) + \Delta_1 (\mu u_{2}) = \lambda \Delta_1 (u_{1}) + \mu \Delta_1(u_{2}),
\end{equation} for all $\lambda,\mu \in \mathbb{R}$.\smallskip 

We shall next show that $\Delta$ preserves real-valued triple transition pseudo-probabi-lities between minimal tripotents, that is, if $e$ and $v$ are two minimal tripotents in $C$ with $TTP (e,v)\in \mathbb{R}$, then $TTP(\Delta (e),\Delta(v)) = TTP(e,v)$. For this purpose, let $e$ and $v$ be minimal tripotents in $C$ with $TTP(e,v)= \varphi_v(e)\in \mathbb{R}$. We have seen in the proof of the previous \cref{thm of extension T and Delta} that if $\varphi_v (e)\neq 0,$ we have $\{\varphi_v (e) v,\varphi_v (e) v, \varphi_v (e) v\} = \{\varphi_v (e) v, e, \varphi_v (e) v\},$ and hence by our hypotheses and \eqref{eq Delta behaves real linearly on real combinations of unitaries}  we obtain $$\begin{aligned}
\varphi_v (e)^2 \overline{\varphi_v (e)} \Delta(v) &= \{\Delta(\varphi_v (e) v),\Delta(\varphi_v (e) v), \Delta(\varphi_v (e) v)\} \\
&= \{\Delta(\varphi_v (e) v), \Delta(e), \Delta(\varphi_v (e) v)\} = \varphi_v (e)^2 \overline{\varphi_{\Delta(v)} (\Delta(e))}\Delta(v), 
\end{aligned} $$ which assures that $TTP(e,v) = \varphi_v (e) = \varphi_{\Delta(v)} (\Delta(e)) = TTP(\Delta(e),\Delta(v)).$ If $TTP(e,v) = \varphi_{v} (e)=0$ the proof is similar via \cref{lem preserve annihil}$(ii)$. \smallskip

The conclusion in \eqref{eq Delta1 fixes min projections S_2} is a bit more elaborated in the case of $M_2$ and we shall need to replace $\Delta_1$ with another mapping. Consider $p_1 = \left( \begin{array}{cc}
	\frac12 & \frac12 \\
	\frac12 & \frac12
\end{array}  \right)\in M_2.$ Since $\Delta_1 (e_1) = e_1$, $\Delta_1 (p_1)$ $ = \left( \begin{array}{cc}
\alpha_1 & \gamma_1\sqrt{\alpha_1 (1-\alpha_1)} \\
\overline{\gamma_1} \sqrt{\alpha_1(1-\alpha_1)} & 1-\alpha_1
\end{array}  \right),$ for some $\alpha_1\in [0,1],$ $\gamma_1\in \mathbb{T},$ and $$TTP( \Delta_1(p_1),\Delta_1(e_1)) = TTP(p_1,e_1) = \frac12,$$ and $\Delta_1(p_1)$ is a minimal projection, we deduce that $\Delta_1 (p_1) =  \left( \begin{array}{cc}
\frac12 & \gamma_1 \frac12 \\
\overline{\gamma_1} \frac12 & \frac12
\end{array}  \right).$ \smallskip

Consider the $^*$-automorphism $\Phi_2: M_2\to M_2$, $\Phi_2 (x) := u x u^*$, with $u= \left( \begin{array}{cc}
\delta_1 & 0 \\
0 &\overline{\delta_1} 
\end{array}  \right),$ where $\delta_1\in \mathbb{T}$ and $\delta_1^2 = \overline{\gamma_1}$. Setting $\Delta_2 = \Phi_2 \Delta_1$, it is not hard to check that $\Delta_2 (e_j) = e_j$ for all $j=1,2$, $\Delta_2 (p_1) = p_1$ and $\Delta_2 (I-p_1) = I-p_1$. Consequently, $$\Delta_2 \left(\left( \begin{array}{cc}
0 & 1 \\
1 &0 
\end{array}  \right)\right) = \Delta_2 (p_1) - \Delta_2 (I-p_1) = 2p_1 -I = \left( \begin{array}{cc}
0 & 1 \\
1 &0 
\end{array}  \right).$$ Take now $p_2 $ $=$ $\left( \begin{array}{cc}
	\frac12 & i\frac12 \\
	-i \frac12 & \frac12
\end{array}  \right).$ Since $\frac12= $ $ TTP(p_2,e_1) $ $= TTP(\Delta_2 (p_2), \Delta_2(e_1)) $ $= TTP(\Delta_2 (p_2), e_1) = \alpha,$ and thus $\Delta_2 (p_2) = \left( \begin{array}{cc}
	\frac12 & \gamma' \frac12 \\
	\overline{\gamma'} \frac12 & \frac12
\end{array}  \right),$ for some $\gamma'\in \mathbb{T}$ (cf. \cref{prop basic properties Jordan}$(b)$). Now the equality $$\begin{aligned}
	\frac{1+ (\gamma'+\overline{\gamma'}) \frac12}{2}&= TTP(\Delta_2(p_2),p_1) =TTP(\Delta_2(p_2),\Delta_2(p_1)) \\
	&= TTP(p_2,p_1) = \frac{1}{2},
\end{aligned}$$ which implies that $\gamma' = \pm i$. \smallskip

We shall distinguish two cases: \smallskip

\noindent\emph{Case 1:} $\Delta_2 (p_2) = \left( \begin{array}{cc}
	\frac12 & i \frac12 \\
	-i \frac12 & \frac12
\end{array}  \right).$ We know that any other minimal projection in $M_2$ can be written in the form  $p $ $=$ $\left( \begin{array}{cc}
	\alpha & \gamma\sqrt{\alpha (1-\alpha)} \\
	\overline{\gamma} \sqrt{\alpha(1-\alpha)} & 1-\alpha
\end{array}  \right),$ for some $\alpha\in (0,1),$ $\gamma\in \mathbb{T}$. Since $\mathbb{R}\ni TTP(p,e_1) = \alpha$, it follows that $TTP(\Delta_2 (p), \Delta_2(e_1))= TTP(\Delta_2 (p), e_1) = \alpha,$ and thus $\Delta_2 (p) = \left( \begin{array}{cc}
\alpha & \gamma' \sqrt{\alpha (1-\alpha)} \\
\overline{\gamma'} \sqrt{\alpha(1-\alpha)} & 1-\alpha
\end{array}  \right),$ for some $\gamma'\in \mathbb{T}$ (cf. \cref{prop basic properties Jordan}). Now the equalities $$\begin{aligned}
\frac{1+ (\gamma'+\overline{\gamma'}) \sqrt{\alpha(1-\alpha)}}{2}&= TTP(\Delta_2(p),p_1) =TTP(\Delta_2(p),\Delta_2(p_1)) \\
&= TTP(p,p_1) = \frac{1+ (\gamma+\overline{\gamma}) \sqrt{\alpha(1-\alpha)}}{2} \\
\frac{1- i (\gamma'-\overline{\gamma'}) \sqrt{\alpha(1-\alpha)}}{2}&= TTP(\Delta_2(p),p_2) =TTP(\Delta_2(p),\Delta_2(p_2)) \\
&= TTP(p,p_2) = \frac{1-i  (\gamma-\overline{\gamma}) \sqrt{\alpha(1-\alpha)}}{2},
\end{aligned}$$ imply that $\gamma =  \gamma'$, witnessing that $\Delta_2$ acts as the identity on the minimal projections of $M_2$.\smallskip

\noindent\emph{Case 2:} $\Delta_2 (p_2) = \left( \begin{array}{cc}
	\frac12 & -i \frac12 \\
	i \frac12 & \frac12
\end{array}  \right).$ Set $\Phi_3:M_2\to M_2$ defined by $\Phi_3(x) = \overline{x}$, the natural component-wise conjugation on $M_2$, which is a conjugate-linear triple automorphism. By replacing $\Delta_2$ with $\Delta_3 (x) := \overline{\Delta_2 (x)} = \Phi_3 \Delta_2 (x)$ ($x\in M_2$), we apply \noindent\emph{Case 1} to $\Delta_3$ to deduce that $\Delta_3$ as the identity on the projections of $M_2$. Observe that the potential extension of $\Delta_3$ will be conjugate-linear.  \smallskip

We have so far shown that there exists a linear or conjugate-linear triple automorphism $\Phi$ on $C$ such that the mapping $\Phi \Delta = \Delta_0$ acts as the identity on all projections of $C$ and $\Delta_0(\lambda e) = \lambda e,$ for all $\lambda \in \mathbb{R}$. \smallskip

For each minimal tripotent $e\in C$ let $f_e^0: \mathbb{C}\to \mathbb{C}$ denote the bijection satisfying $\Delta_0 (\lambda e) = f_e^0 (\lambda) \Delta_0(e)$, for all $\lambda \in \mathbb{C}$ (see page~\pageref{def of fe}). We claim that 
\begin{equation}\label{eq f_e does not depend on e} f_p^0 = f_{e_1}^0= f_{e_2}^0, \hbox{ for every minimal projection } p\in C.
\end{equation} Suppose $ p = \left( \begin{array}{cc}
\alpha & \gamma\sqrt{\alpha (1-\alpha)} \\
\overline{\gamma} \sqrt{\alpha(1-\alpha)} & 1-\alpha
\end{array}  \right),$  ($\alpha\in [0,1],$ $\gamma\in \mathbb{T}$) is a minimal projection in $C$ ---in case that $C= S_2$ we have $\gamma =1$. Observe that by \eqref{eq additivity of Delta on Ce1 + Ce2} and the properties of $\Delta_0$ we have 
$$\begin{aligned}
f_{e_1}^0 (\lambda) e_1 + f_{e_2}^0 (\lambda) e_2 &= f_{e_1}^0 (\lambda) \Delta_0 (e_1) + f_{e_2}^0 (\lambda)\Delta_0(e_2) = \Delta_0 (\lambda (e_1+ e_1)) = \Delta_0 (\lambda I) \\
&= \Delta_0 (\lambda p + \lambda (I-p)) =f_{p}^0 (\lambda) \Delta_0 (p) + f_{I-p}^0 (\lambda)\Delta_0(I-p) \\
&= f_{p}^0 (\lambda) p + f_{I-p}^0 (\lambda) (I-p).
\end{aligned}$$ The arbitrariness of $p$ (respectively, of $\alpha$ and $\gamma$) implies $f^0_{e_1} =f^0_{e_2} = f^0_{p}$ for every minimal projection $p$. This finishes the proof of the claim. We can therefore write $f^0$ for any of the mappings $f_{p}^0$. We also know that $f^0$ is the identity on $\mathbb{R}$.\smallskip

Every minimal tripotent $v\in M_2$ admits the following matrix representation:
\begin{equation}\label{alpha beta delta}
	v = \left( \begin{array}{cc}
		\alpha & \beta \\
		\gamma & \delta
	\end{array}  \right), \quad \text{satisfying} \begin{cases}
		|\alpha|^2 +|\delta|^2 + |\beta|^2 + |\gamma|^2 = 1,\\
		\alpha \delta - \beta \gamma = 0,
	\end{cases}
\end{equation} where $\alpha, \beta, \gamma, \delta \in \CC$ (cf. \cite[Lemma 3.10]{Polo_Peralta_AdvMath_2018}) ---in $S_2$ we have $\gamma = \beta$. Consider the next minimal projection $p_1= \left( \begin{array}{cc}
	1/2 & 1/2 \\
	1/2 & 1/2
\end{array}  \right) $. We must have $\Delta_0 (v) =   \left( \begin{array}{cc}
	\tilde{\alpha} & \tilde{\beta} \\
	\tilde{\gamma} & \tilde{\delta}
\end{array}  \right)$ for some $\tilde{\alpha}, \tilde{\beta}, \tilde{\delta}, \tilde{\gamma} \in \CC$ satisfying analogous conditions to those in \cref{alpha beta delta} (cf. \cref{prop basic properties Jordan}).\smallskip

We begin by proving that $\tilde{\alpha} = f^0 (\alpha),$ equivalently, $\Re\hbox{e}(f^0(\alpha)) = \Re\hbox{e}(\alpha)$, $\Im\hbox{m}(f^0(\alpha))$ $= \Im\hbox{m}(\alpha)$. Namely, since $\alpha e_{1}$ and $\delta e_{2}$ are truncations of $v$ and $\Delta_0$ preserves the truncation of tripotents (cf. \cref{r Delta preserves the truncation of tripotents}), it follows that $\tilde{\alpha}=f^0 (\alpha)$ and $\tilde{\delta} =f^0 (\delta)$.\smallskip

Suppose that $\alpha +\beta +\gamma +\delta \neq 0$. Since $\Delta_0(p_1) = p_1,$ and $P_{2}(p_1)(v) = \frac{\alpha +\beta+\gamma +\delta}{2} p_1$, we obtain that $\frac{\alpha +\beta+\gamma  +\delta}{2}p$ is a truncation of $v$ and 
\begin{equation}\label{equation f0}
	f^0\left(\frac{\alpha +\beta +\gamma  +\delta}{2}\right) = \frac{\tilde{\alpha}+\tilde{\beta} +\tilde{\gamma} + \tilde{\delta}}{2} = \frac{f^0(\alpha)+\tilde{\beta}+ \tilde{\gamma} +f^0(\delta)}{2}.  
\end{equation}
In particular, \eqref{equation f0} holds for the minimal tripotents
\begin{align*}
	v_1=\frac{1}{4}\left(\begin{array}{cc}
		\alpha & 1 \\
		1 & \overline{\alpha}
	\end{array}  \right), \mbox{ and }  v_2=\frac{1}{4}\left(\begin{array}{cc}
		-\alpha & 1 \\
		1 & -\overline{\alpha}
	\end{array}  \right) \hbox{ (both in $S_2$) }
\end{align*} where $|\alpha| = 1$, and their images $\Delta_0 (v_j) =   \left( \begin{array}{cc}
\tilde{\alpha}_j & \tilde{\beta}_j \\
\tilde{\gamma}_j & \tilde{\delta}_j
\end{array}  \right)$ ($j=1,2$). Since $v_1 \perp v_2$, it follows from \eqref{eq additivity of Delta on Ce1 + Ce2 two} and \eqref{eq Delta behaves real linearly on real combinations of unitaries} that $$\begin{aligned} f^0\left(\frac12 \alpha\right) e_1+ \overline{f^0\left(\frac12 {\alpha}\right)} e_2 &= 
f^0\left(\frac12 \alpha\right) e_1+ f^0\left(\frac12 \overline{\alpha}\right) e_2= \Delta_0 \left( \frac12 \left(\begin{array}{cc}
	\alpha & 0 \\
	0 & \overline{\alpha}
\end{array}  \right) \right)   \\
&= \Delta_0 (v_1 -v_2)= \Delta_0 (v_1) - \Delta_0 (v_2),
\end{aligned}$$ which implies that $\tilde{\beta}_1 = \tilde{\beta}_2$ and $\tilde{\gamma}_1 = \tilde{\gamma}_2$. Therefore, a simple computation with the identity in \eqref{equation f0} gives $\Re\hbox{e}(f^0(\alpha)) = \Re\hbox{e}(\alpha),$ for all $\alpha \in \mathbb{T}$. \smallskip

By \cref{cor of fmultiplicative} $f^0 (i) \in \{\pm i\}$, and thus 
$$ -\Im\hbox{m}(\alpha) =\Re\hbox{e}(f^0(i\alpha)) = \Re\hbox{e}((\pm i) f^0(\alpha)) = \mp \Im\hbox{m}(f^0(\alpha)),$$ for all $\alpha \in \mathbb{T}$. The properties of $f^0$ guarantee that $f^0 (\lambda) = \lambda $ or $f^0 (\lambda) = \overline{\lambda}$ for all $\lambda \in \mathbb{C}$. Up to composing with the ``natural'' (entry-wise) conjugation on $C$, we can always assume that $f^0 (\lambda) = \lambda $ for all $\lambda \in \mathbb{C}$. \smallskip

In our penultimate step we prove that, under the previous assumptions, $f_e^0 (\lambda) =\lambda $ for all  $\lambda \in \mathbb{C}$ and every minimal tripotent $e$ in $C$. This is now an easy consequence of what we proved above, just observe that for each minimal tripotent $e\in C$ we can find another minimal tripotent $\tilde{e}$ which is orthogonal to $e$, and hence $u=e+\tilde{e}$ is a unitary in $C$. Every unitary can be written in the form $u = \gamma_1 p + \gamma_2 (I-p),$ where $p$ is a minimal projection in $C$ and $\gamma_j\in \mathbb{C}$.  We deduce from all the previous properties that $$\begin{aligned} f_e^0(\lambda) \Delta_0 (e ) + f_{\tilde{e}}^0 (\lambda) \Delta_0 (\tilde{e})&= \Delta_0 (\lambda (e+\tilde{e}) )= \Delta_0(\lambda u) = \Delta_0 (\lambda (\gamma_1 p + \gamma_2 (I-p)) ) \\
&= f^0 (\lambda \gamma_1)  p + f^0(\lambda \gamma_2)  (I-p) = \lambda \Delta_0 (\gamma_1 p + \gamma_2 (I-p) ) \\
&= \lambda \Delta_0 (u)  = \lambda \Delta_0 (e ) + \lambda\Delta_0 (\tilde{e}), 
\end{aligned}$$ which leads to $f_e^0 (\lambda)= \lambda$ for all complex $\lambda$. \smallskip

Finally, the restriction of $\Delta_0$ to the set of all tripotents of $C$ is a bijection preserving order and orthogonality in both directions (cf. \cref{lem pres tripotents and leq} and \cref{lem orthogonal tripotents}). We also have $\Delta_0 (\lambda e) = \lambda \Delta_{0} (e)$ for every minimal tripotent $e\in C$ and every complex $\lambda$. Therefore, by \cite[Theorem 4.3]{Fried_Peralta_Ann_Math_Phys_2022}, applied to $\Delta_0|_{\mathcal{U}(C)} : \mathcal{U}(C)\to \mathcal{U}(C),$ there exists a linear triple automorphism $T_0:C\to C$ satisfying $T_0 (v) = \Delta_0(v)$ for all tripotent $v\in C$. We are in a position to apply the identity principle in \cref{thm of extension T and Delta} to deduce that $\Delta_0 = T_0$ is a linear triple automorphism. It follows that $\Delta$ itself is a linear or a conjugate-linear triple automorphism on $C$.
\end{proof}

\section{Preservers of the truncation of triple products on atomic \texorpdfstring{JBW$^*$}{}-triples}\label{sec: main conclusions}

In this final section we shall study bijections preserving the truncation of triple products between atomic JBW$^*$-triples. Let us fix two atomic JBW$^*$-triples $A=\bigoplus\limits_{i\in \Gamma_1}^{\ell_{\infty}} C_i$ and $B=\bigoplus\limits_{j\in \Gamma_2}^{\ell_{\infty}} D_j$ where $C_i$, and $D_j$ are Cartan factors, and a {\rm(}non-necessarily linear{\rm)} bijection $\Delta :A\to B$ preserving the truncation of triple products in both directions.\smallskip

Our first comment is to note that one-dimensional Cartan factors must be excluded. Namely, we can consider standard counter-examples (\cite{FriHak1988,Chabb2017, ChabbMbekhta2017,EssPe2018}) as $\Delta: \mathbb{C}\to \mathbb{C}$, $\Delta (\lambda) = \lambda^{-1}$ if $\lambda \neq 0$ and $\Delta(0)=0,$ is a bijection preserving the truncation of triple products in both directions, which is not additive. We shall therefore assume that both $A$ and $B$ do not contain 1-dimensional factors.\smallskip 

Fix $i\in \Gamma_1$. Observe that, by minimality, each minimal tripotent $e\in C_i$ belongs to a unique Cartan factor in $B$, that is, there exists $j\in \Gamma_2$ such that $\Delta (e)\in D_j$ (cf. \cref{lem pres tripotents and leq}). For any other $v\in \mathcal{U}_{min} (C_i)$ there exists $w\in \mathcal{U}_{min} (C_i)$ such that $e,v\not\perp w$ (cf. \cite[Lemma 3.10]{Polo_Peralta_AdvMath_2018} when the rank is $\geq 2$, while in the rank-one case, i.e. Hilbert spaces the statement is even more clear), and hence $\Delta(e),\Delta(v)\not\perp \Delta (w),$ which is possible if and only if $\Delta(v), \Delta (w)\in D_j$. Therefore, there exists a unique $\sigma(i)\in \Gamma_2$ such that $\Delta (\mathcal{U}_{min} (C_i)) = \mathcal{U}_{min} (D_{\sigma(i)})$. Clearly, the mapping $\sigma: \Gamma_1 \to \Gamma_2$ is a bijection. Furthermore, if we take a non-zero $a\in C_i$, there is a minimal tripotent $e\in C_i$ admitting a non-zero scalar multiple which is a truncation of $a$, and thus the same occurs to $ \Delta (e)$ and $\Delta (a)$. However, for any other minimal tripotent $\Delta (w) \in  \mathcal{U}_{min}(D_{\sigma(i_1)}) = \Delta\left(\mathcal{U}_{min}(C_{i_1})\right),$ we have $\{w,a,w\} =0$ and thus $\{\Delta(w),\Delta(a),\Delta(w)\} =0$ (see \cref{lem preserve annihil}$(ii)$), which is only possible if $\Delta (a)\in D_{\sigma(i)}$. We have actually shown that $\Delta (C_i) = D_{\sigma(i)}$. Observe that \cref{lem orthogonal tripotents} assures that the rank of $C_i$ is $\geq n$ if and only if the rank of $D_{\sigma(i)}$ is. \smallskip 

We gather the previous conclusion in the next lemma. 

\begin{lem}\label{l Delta maps a factor onto a factor} There exists a bijection $\sigma: \Gamma_1\to \Gamma_2$ satisfying that for each $i\in \Gamma_1$, $\Delta$ maps $C_i$ onto $D_{\sigma(i)}$. Moreover, the rank of $C_i$ is greater than or equal to $n$ if and only if the same happens to the rank of $D_{\sigma(i)}$.
\end{lem}

For each minimal tripotent $e\in A$ we shall denote by $f_e: \mathbb{C}\to \mathbb{C}$ the bijection defined in page~\pageref{def of fe} satisfying $\Delta (\lambda e) = f_e (\lambda) \Delta(e)$ for all complex $\lambda$. We shall see next that if $e$ belongs to Cartan factor in $A$ with rank $\geq 2$, the mapping $f_e$ is just the identity or the conjugation on $\mathbb{C}$, and actually tripotents in the same factor share the same mapping $f_e$.

Let us recall some definitions from \cite{Dang_Friedm_MathScand_1987}. Given two tripotents $u,v$ in a JB$^*$-triple $E$, we say that $u$ is \textit{colinear} to $v$ ($u \top v$ in short) if $u \in E_{1}(v)$ and $v \in E_{1}(u)$. In case that $u \in E_{1}(v)$ and $v \in E_{2}(u)$ we say $u$ \textit{governs} $v$ ($u \vdash v$ in short). An ordered triplet $(v, u, \tilde{v})$ of tripotents in $E$ is called a \textit{trangle} if $v \perp \tilde{v}$, $u \vdash v,\tilde{v}$ and $\tilde{v} = Q(u) (v)$.  An ordered quadruple $(v_{1}, v_{2}, v_{3}, v_{4})$ is called a \textit{quadrangle} if $v_{1} \perp v_{3}$, $v_{2} \perp v_{4}$, $v_{1} \top v_{2} \top v_{3} \top v_{4} \top v_{1}$ and $\{ v_{1}, v_{2},v_{3}\} = \frac{1}{2} v_{4}$ (note that, by the Jordan identity \eqref{Jordan identity}, the last equation holds whenever the indices are permuted cyclically).
\smallskip

\begin{prop}\label{p fe is idenity or conjugation and all coincide on eeach factor} Let $e$ be a minimal tripotent belonging to a Cartan factor of rank $\geq 2$ in $A$. The following statements hold:
\begin{enumerate}[$(a)$]
\item Then $f_e (\lambda) = \lambda $ for all $\lambda\in \mathbb{C},$ or $f_e (\lambda) = \overline{\lambda} $ for all $\lambda\in \mathbb{C}.$  
\item If $v$ is another minimal tripotent in $A$ belonging to the same Cartan factor containing $e$, we have $f_e = f_v$.
\end{enumerate}
\end{prop}

\begin{proof} Let us try to prove both statements at the same time. So, we pick minimal tripotents $e,v\in C_i$ and we assume that $C_i$ has rank $\geq 2$. To simplify the notation we shall write $C$ for $C_i$ and $D$ for $D_{\sigma(i)}$. By \cite[Lemma 3.10]{Polo_Peralta_AdvMath_2018} one of the next facts holds: 
\begin{enumerate}[$(i)$]
\item There exist minimal tripotents $v_{2}, v_{3},v_{4}$ in $C$ and complex numbers $\alpha, \beta, \gamma, \delta$ such that $(e, v_{2}, v_{3},v_{4})$ is a quadrangle, $|\alpha|^2 + |\beta|^2 + |\gamma|^2 +|\delta|^2= 1$, $\alpha \delta = \beta \gamma$ and $v= \alpha e+ \beta v_{2} + \gamma v_{4} +\delta v_{3}.$
\item There exist a minimal tripotent $v_4 \in C$, a rank two tripotent $u \in A$, and complex numbers $\alpha, \beta, \delta$ such that $(e,u, v_4)$ is a trangle, $|\alpha|^2 + 2|\beta|^2  +|\delta|^2= 1$, $\alpha \delta = \beta^2$, and $v = \alpha e + \beta u + \delta \tilde{e}$. 
\end{enumerate} 
The tripotents $e+v_3\in C$ and $\Delta(e+v_3) = \Delta(e)+ \Delta (v_3)\in D$ (see \cref{lem orthogonal tripotents}). As we have commented in \cref{subsec:spin factors}, the Peirce-2 subspaces $C_2(e+v_3)$ and $D_2 (\Delta(e+v_3))$ are two spin factors (cf. \cite[Lemma 3.6]{Kal_Peralta_Ann_Math_Phys_2021}). \cref{Peirce2toPeirce2} assures that $\Delta|_{C_2(e+v_3)} :C_2(e+v_3)\to D_2 (\Delta(e+v_3))$ is a bijection preserving the truncation of triple products in both directions. Since $e,v\in C_2(e+v_3)$, we can apply \cref{r preservers of truncations between spin factors}$(d)$ to find two JB$^*$-subtriples $M\subseteq C_2(e+v_3)$, $N\subseteq  D_2 (\Delta(e+v_3))$, $e,v \in M,$ $\Delta(M) =N$, and $M,N\cong M_2$ or $M,N\cong S_2$.  In this case \cref{linear at identity real} proves that $\Delta|_{M}$ is linear or conjugate linear, and consequently, $\Delta (\lambda e) = \lambda \Delta(e)$ and  $\Delta (\lambda v) = \lambda \Delta(v)$ for all complex $\lambda,$ or $\Delta (\lambda e) = \overline{\lambda} \Delta(e)$ and  $\Delta (\lambda v) = \overline{\lambda} \Delta(v)$ for all complex $\lambda$, which concludes the proof.
\end{proof}

Thanks to the previous proposition we can classify the Cartan factors in $A$'s decomposition in the following way: $\Gamma_{1,rk1} = \{i\in \Gamma_1 :  \hbox{ rank } C_i =1 \},$
$$\Gamma_{1,l} = \{i\in \Gamma_1 : f_e (\lambda) = \lambda \hbox{ for all } \lambda \in \mathbb{C}, e\in \mathcal{U}_{min}(C_i), \hbox{ rank } C_i \geq 2\}$$ and $$\Gamma_{1,cl}  = \{i\in \Gamma_1 : f_e (\lambda) = \overline{\lambda} \hbox{ for all } \lambda \in \mathbb{C}, e\in \mathcal{U}_{min}(C_i), \hbox{ rank } C_i \geq 2\}.$$ Clearly, $C_i$ and $D_{\sigma(i)}$ have rank-one for all $i\in \Gamma_{1,rk1}$, and $\Delta (\lambda e) = \lambda \Delta(e)$ (respectively, $\Delta (\lambda e) = \overline{\lambda} \Delta(e)$) for all $\lambda\in \mathbb{C},$ and every minimal tripotent $e\in \bigoplus\limits_{i\in \Gamma_{1,l}}^{\ell_{\infty}} C_i$ \Big(respectively, $e\in \bigoplus\limits_{i\in \Gamma_{1,cl}}^{\ell_{\infty}} C_i$\Big).\smallskip

Dealing with rank-one Cartan factors requires additional hypotheses. It is known that a Cartan factor has rank-one if and only if it is a Hilbert space regarded as a type 1 Cartan factor. We have also commented that we must avoid the 1-dimensional factors. 

\begin{prop}\label{p Hilbert spaces rank1CF} Let $\Delta: H\to K$ be a {\rm(}non-necessarily linear{\rm)} bijection preserving the truncation of triple products in both directions, where $H$ and $K$ are two complex Hilbert spaces of dimension $\geq 2$ regarded as type 1 Cartan factors. Suppose additionally that $\Delta$ is continuous. Then $\Delta$ is a linear or conjugate-linear isometry preserving triple products.   
\end{prop}

\begin{proof} The triple product in $H$ is given by $\{a,b,c\} = \frac12 \langle a| b\rangle c+  \frac12 \langle c| b\rangle a$ ($a,b,c\in H$). Every non-zero tripotent $e$ in $H$ is minimal and maximal, they actually coincide with the points in the unit sphere of $H$, $H_2(e) = \mathbb{C} e$, $H_1(e) = \{e\}^{\perp_2}_{H}$, $H_0(e) = \{0\}.$  By \cref{c preservation Perice 1 + 0} we know that $\Delta (\{e\}^{\perp_2}_{H}) = \{\Delta(e)\}^{\perp_2}_{H}$ for every minimal tripotent $e\in H$, where $\Delta(e)$ also is a minimal tripotent in $K$ (see \cref{lem pres tripotents and leq}).\smallskip
	
Let $\{e_j: j\in \Lambda\}$ be an orthonormal basis of $H$. By the properties of $\Delta$, the set $\{\Delta(e_j) : j\in \Lambda\}$ is an orthonormal basis of $K$ (cf. \cref{lem pres tripotents and leq} and \cref{lem orthogonal tripotents} and \cref{c preservation Perice 1 + 0}). 	The continuity of $\Delta$ also implies that, for each minimal tripotent $e\in H$, the mapping $f_e:\mathbb{C}\to \mathbb{C}$ is a continuous bijection preserving products and conjugation (see \cref{cor of fmultiplicative}). It is known that under these circumstances $f_{e} (\lambda) = \lambda$ for all $\lambda\in \mathbb{C}$ or $f_{e} (\lambda) = \overline{\lambda}$ for all $\lambda\in \mathbb{C}$ \cite[Section 4.2.1]{DymMcKeanBook1972}.\smallskip

Take a norm-one element $a\in H.$ Clearly, $a$ is a tripotent, and hence $\Delta(a )$ also is a non-zero tripotent (cf. \cref{lem pres tripotents and leq}), hence $\|\Delta(a)\|=\|a\| =1.$  Suppose $e$ is another norm-one element in $H$. If $\langle a , e\rangle =0$ we have $\Delta(e) \in \{\Delta(a)\}^{\perp_2}_{K}$. If $\langle a , e\rangle \neq 0$, the element $|\langle a, e\rangle|^{-2} {\langle a, e\rangle} e = (\overline{ \langle a, e\rangle } )^{-1} e$ is a truncation of $a$ (which is a tripotent). Therefore, by \cref{r Delta preserves the truncation of tripotents} we deduce that  
$\Delta \left(|\langle a, e\rangle|^{-2} {\langle a, e\rangle} e \right)$ is a truncation of $\Delta(a)$, that is, $\{\Delta(a), \Delta \left(|\langle a, e\rangle|^{-2} {\langle a, e\rangle} e \right) ,\Delta (a)\} = \Delta(a)$. \smallskip

Suppose we are in the case that $f_e(\lambda) = \lambda$ for all $\lambda \in\mathbb{C}$, for each $a\in H$ with $\langle a,e\rangle \neq 0$, we get 
$$ \Delta(a)=  |\langle a, e\rangle|^{-2} \overline{\langle a, e\rangle}  \{\Delta(a), \Delta \left( e \right) ,\Delta (a)\} = {{\langle a, e\rangle}}^{-1} \langle \Delta(a), \Delta(e)\rangle  \Delta(a),$$ which shows that  $\langle \Delta(a) | \Delta(e)\rangle = \langle a | e\rangle$ for every norm-one elements $a,e$ in $H$ with $\langle a,e\rangle \neq 0$. Actually the latter condition can be relaxed by \cref{c preservation Perice 1 + 0}. Thus $\langle \Delta(a) | \Delta(e)\rangle = \langle a | e\rangle$ for every norm-one elements $a,e$ in $H$, and hence  $\langle \Delta(a) | \Delta(e)\rangle = \langle a | e\rangle$ for every  $a,e$ in $H$.\smallskip

We similarly deduce that $\langle \Delta(a) | \Delta(e)\rangle = \langle e | a\rangle$ for every $a,e$ in $H$ if $f_{\frac{e}{\|e\|}} (\lambda)= \overline{\lambda}$ ($\forall \lambda\in \mathbb{C}$). \smallskip

Set $\Lambda_1:=\{i\in \Lambda : f_{e_i} (\lambda) = \lambda, \ \forall \lambda \}$ and $\Lambda_2:=\{i\in \Lambda : f_{e_i} (\lambda) = \overline{\lambda}, \ \forall \lambda \}$. Since $\{e_j\}_{j\in \Lambda}$ and $\{\Delta(e_j)\}_{j\in \Lambda}$ are orthonormal basis in $H$ and $K$, respectively, we have $$\Delta(a) = \sum_{i\in\Lambda} \langle \Delta(a)| \Delta(e_i)\rangle \Delta (e_i) = \sum_{i\in\Lambda_1} \langle a| e_i\rangle \Delta(e_i) + \sum_{i\in\Lambda_2} \langle e_i| a\rangle \Delta(e_i) \ (\forall a\in H).$$ If $\Lambda_1,\Lambda_2\neq \emptyset$, we can pick $i\in \Lambda_1$ and $j \in \Lambda_2$. The element $e= \frac{e_i+e_j}{\sqrt{2}}$ is a minimal tripotent in $H$, and hence $$\Delta (\lambda e) = \lambda \Delta(e) \ (\forall\lambda), \hbox{ or } \Delta (\lambda e) = \overline{\lambda} \Delta(e) \ (\forall\lambda).$$ In each one of the cases we have $$\frac{\lambda}{\sqrt{2}} \Delta(e_i) + \frac{\overline{\lambda}}{\sqrt{2}} \Delta(e_j) =\Delta (\lambda e)= \lambda \Delta(e) = \frac{\lambda}{\sqrt{2}} \Delta(e_i) + \frac{{\lambda}}{\sqrt{2}} \Delta(e_j)$$ or $$\frac{\lambda}{\sqrt{2}} \Delta(e_i) + \frac{\overline{\lambda}}{\sqrt{2}} \Delta(e_j) =\Delta (\lambda e)= \overline{\lambda} \Delta(e) = \frac{\overline{\lambda}}{\sqrt{2}} \Delta(e_i) + \frac{\overline{\lambda}}{\sqrt{2}} \Delta(e_j),$$ and both identities are impossible. Therefore, $\Lambda_2=\emptyset$ and thus $\Delta$ is a surjective linear isometric triple isomorphism, or $\Lambda_1=\emptyset$ and thus $\Delta$ is a surjective conjugate-linear isometric triple isomorphism.
\end{proof}

We are finally in a position to present our main result, which gives an affirmative answer to \cref{prob1} on atomic JBW$^*$-triples non-containing $1$-dimensional Cartan factor. We further show that $\Delta$ is automatically continuous and actually a real linear triple (isometric) isomorphism when both atomic JBW$^*$-triples contain no rank-one Cartan factors. Moreover, and the same conclusion holds if $\Delta$ is assumed to be continuous on each rank-one Cartan factor of the atomic JBW$^*$-triple in the  domain.

\begin{thrm}\label{t main thrm preservers of truncations of triple products}
Let $\Delta :A \to B$ be a {\rm(}non-necessarily linear nor continuous{\rm)} bijection preserving the truncation of triple products in both directions, where $A$ and $B$ are atomic JBW$^*$-triples non-containing $1$-dimensional Cartan factors. Assume additionally that the restriction of $\Delta$ to each rank-one Cartan factor in $A$, if any, is a continuous mapping.
Then $\Delta$ is an isometric real linear triple isomorphism. Furthermore, we can decompose $A$ as the direct $\ell_{\infty}$-sum of two (possibly trivial) weak$^*$-closed subtriples $A=A_1\oplus^{\infty} A_2$ such that $\Delta|_{A_1}$ is complex linear and $\Delta|_{A_2}$ is conjugate linear.  
\end{thrm}

\begin{proof} We keep the notation we have been employing along this section. For each $i \in \Gamma_{1,rk1}$, $\Delta|_{C_{i}} : C_{i}\to D_{\sigma(i)}$ is a bijection preserving truncation of triple products in both directions between two rank-one Cartan factors (cf. \cref{l Delta maps a factor onto a factor}). \cref{p Hilbert spaces rank1CF} implies that $\Delta|_{C_{i}}$ is a linear or a conjugate-linear triple isomorphism, and this holds for all $i \in \Gamma_{1,rk1}$. This proves that $\Delta|_{\bigoplus_{i \in \Gamma_{1,rk1}} C_{i}}$ is a real linear triple isomorphism.\smallskip
	 
We set $A_{l} = \bigoplus\limits_{i\in \Gamma_{1,l}}^{\ell_{\infty}} C_i,$ $A_{cl} = \bigoplus\limits_{i\in \Gamma_{1,cl}}^{\ell_{\infty}} C_i$, $B_{l} = \bigoplus\limits_{i\in \Gamma_{1,l}}^{\ell_{\infty}} D_{\sigma(i)}$ and $B_{cl} = \bigoplus\limits_{i\in \Gamma_{1,cl}}^{\ell_{\infty}} D_{\sigma(i)}$. Clearly $\Delta\left(A_{l}\right) = B_{l}$ and $\Delta\left(A_{cl}\right) = B_{cl}$ (cf. \cref{l Delta maps a factor onto a factor}).\smallskip

Let $\overline{B_{cl}}$ denote the complex atomic JBW$^*$-triple obtained from $B_{cl}$ by replacing the complex structure with the conjugate one and keeping all the other operations. We shall write $\lambda \cdot {x} := \overline{\lambda} {x}$ ($\lambda\in \mathbb{C}$, ${x}\in  \overline{B_{cl}}$). Then the mapping $\Delta_1 : A_{l}\oplus^{\infty} A_{cl}\to B_{l}\oplus^{\infty} \overline{B_{cl}},$ $\Delta_1(a) = \Delta(a)$, is a bijection preserving the truncation of triple products in both directions, and thanks to the change in the complex structure of $\overline{B_{cl}}$, we have $\Delta_1 (\lambda e) = \lambda \cdot \Delta_1(e)$ for all complex $\lambda$ and all minimal tripotent $e\in A_{cl}$. Clearly, the same holds for the original structure of $A_{l}$ and $B_{l}$.\smallskip

By \cref{lem pres tripotents and leq} and  \cref{lem orthogonal tripotents} the restricted mapping $$\Delta_1|_{\mathcal{U}( A_{l}\oplus^{\infty} A_{cl})} : \mathcal{U}{( A_{l}\oplus^{\infty} A_{cl})} \to \mathcal{U}{(B_{l}\oplus^{\infty} \overline{B_{cl}})}$$ is a bijection preserving the partial ordering and orthogonality between tripotents in both directions. 
We also know that $\Delta_1 (\lambda e) = \lambda \cdot \Delta_1(e)$ (respectively, $\Delta_1 (\lambda e) = \lambda \Delta_1(e)$) for all complex $\lambda\in \mathbb{T}$ and all minimal tripotent $e\in A_{cl}$ (respectively, $e\in A_{l}$). By \cite[Theorem 6.1]{Fried_Peralta_Ann_Math_Phys_2022} there exists a linear triple isomorphism $T:  A_{l}\oplus^{\infty} A_{cl}\to  B_{l}\oplus^{\infty} \overline{B_{cl}}$ whose evaluation at every (minimal) tripotent in $ A_{l}\oplus^{\infty} A_{cl}$ coincides with that of $\Delta_1$.  The identity principle in \cref{thm of extension T and Delta} assures that $\Delta (a) = T(a)$ for all $a\in A =A_{l}\oplus^{\infty} A_{cl}$. We finally observe that when $T$ regarded as a mapping from $A_{l}\oplus^{\infty} A_{cl}$ onto $B_{l}\oplus^{\infty} B_{cl}$ it is only a real linear mapping and $\Delta (a) = T(a),$ for all $a\in A_{l}\oplus^{\infty} A_{cl}$, that is, $\Delta|_{A_{l}\oplus^{\infty} A_{cl}}$ is real linear.
\end{proof}\medskip

\textbf{Acknowledgements} J.J. Garcés supported by grant PID2021-122126NB-C31 funded by MICIU/AEI/10.13039/501100011033 and by ERDF/EU, and by Junta de Andalucía grant FQM375. 
Li supported by National Natural Science Foundation of China (Grant No. 12171251). A.M. Peralta supported by grant PID2021-122126NB-C31 funded by MICIU/AEI/10.13039/501100011033 and by ERDF/EU, by Junta de Andalucía grant FQM375, IMAG--Mar{\'i}a de Maeztu grant CEX2020-001105-M/AEI/10.13039/501100011033 and (MOST) Ministry of Science and Technology of China grant G2023125007L. S. Su supported by grant PID2021-122126NB-C31 funded by MICIU/AEI/10.13039/ 501100011033 and by China Scholarship Council\hyphenation{Council} Program (Grant No.202306740016).\smallskip

Part of this work was completed during a visit of A.M. Peralta to Nankai University and the Chern Institute of Mathematics, which he thanks for the hospitality.\medskip

\textbf{Disclosure:} All authors declare that they have no conflicts of interest to disclose.



\end{document}